\documentclass[journal]{IEEEtran}
\pdfoutput=1
\usepackage{cite}
\usepackage{amssymb}
\usepackage{amsmath}
\usepackage{amsthm}

\usepackage{algorithmic}
\usepackage{graphicx}
\usepackage{textcomp}

\usepackage{hyperref}
\usepackage{caption}
\usepackage{subfigure}
\usepackage[ruled,vlined]{algorithm2e}

\newcommand\br{\mathbf{r}}

\newcommand\bq{\mathbf{q}}
\newcommand\bQ{\mathbf{Q}}
\newcommand\bJ{\mathbf{J}}
\newcommand\bE{\mathbf{E}}
\newcommand\bg{\mathbf{g}}

\newcommand\bx{\mathbf{x}}
\newcommand\bW{\mathbf{W}}
\newcommand\bd{\mathbf{d}}
\newcommand\bH{\mathbf{H}}

\newcommand\oB{\overline{\mathbf{B}}}
\newcommand\ox{\overline{\mathbf{x}}}
\newcommand\og{\overline{\mathbf{g}}}
\newcommand\od{\overline{\mathbf{d}}}
\newcommand\oH{\overline{\mathbf{H}}}

\newcommand\diag{\operatorname{diag}}
\newtheorem{rmk}{Remark}
\newtheorem{fct}{Fact}
\newtheorem{thm}{Theorem}
\newtheorem{lem}{Lemma}
\newtheorem{coro}{Corollary}
\newtheorem{props}{Proposition}
\newtheorem{asp}{Assumption}

\begin{document}
\title{A Communication-Efficient Decentralized Newton's Method with Provably Faster Convergence}
\author{Huikang Liu, Jiaojiao Zhang, Anthony Man-Cho So, and Qing Ling
\thanks{Huikang Liu is with the Research Institute for Interdisciplinary Sciences, School of Information Management and Engineering, Shanghai University of Finance and Economics (e-mail: liuhuikang@sufe.edu.cn). }
\thanks{Jiaojiao Zhang is with the Division of Decision and Control Systems, School of Electrical Engineering and Computer Science, KTH Royal Institute of Technology (e-mail: zdhzjj@mail.ustc.edu.cn).}
\thanks{Anthony Man-Cho So is with the Department of Systems Engineering and Engineering Management, The Chinese University of Hong Kong (e-mail: manchoso@se.cuhk.edu.hk).}
\thanks{Qing Ling is with the School of Computer Science and Engineering and Guangdong Province Key Laboratory of Computational Science, Sun Yat-Sen University, and also with the Pazhou Lab (e-mail: lingqing556@mail.sysu.edu.cn).}}

\maketitle

\begin{abstract}
In this paper, we consider a strongly convex finite-sum minimization problem over a decentralized network
and propose a communication-efficient decentralized Newton's method for solving it.
The main challenges in designing such an algorithm come from three aspects: (i) mismatch between local gradients/Hessians and the global ones; (ii) cost of sharing second-order information; (iii) tradeoff among computation and communication. To handle these challenges, we first apply dynamic average consensus (DAC) so that each node is able to use a local gradient approximation and a local Hessian approximation to  track the global gradient and Hessian, respectively. Second, since exchanging Hessian approximations is far from communication-efficient, we require the nodes to exchange the compressed ones instead and then apply an error compensation mechanism to correct for the compression noise. Third, we introduce  multi-step consensus for exchanging local variables and local gradient approximations to balance between computation and communication. 
%
With novel analysis, we establish the globally linear (resp., asymptotically super-linear) convergence rate of the proposed method when $m$ is constant (resp., tends to infinity), where $m\ge1$ is the number of consensus inner steps. To the best of our knowledge, this is the first super-linear convergence result for a communication-efficient decentralized Newton's method. Moreover, the rate we establish is provably faster than those of first-order methods. Our numerical results on various applications corroborate the theoretical findings.

\end{abstract}

\section{Introduction}
In this paper, we consider solving a finite-sum optimization problem defined over an undirected, connected network with $n$ nodes:
\begin{align}\label{eq-obj}
x^*=\arg\min_{x \in \mathbb{R}^d} F(x) \triangleq \frac{1}{n}\sum_{i = 1}^{n} f_i(x).
\end{align}
Here, $x\in\mathbb{R}^d$ is the decision variable and $f_i: \mathbb{R}^d\to \mathbb{R}$ is a twice-continuously differentiable function privately owned by node $i$. The entire objective function $F$ is assumed to be strongly convex. Each node  is allowed to exchange limited information with its neighbors during the optimization process. To make \eqref{eq-obj} separable across the nodes, one common way is to introduce a local copy  $x_i\in\mathbb{R}^d$ of $x$ for node $i$ and then force all the local copies to be equal by adding consensus constraints.  
This leads to the following alternative formulation of Problem  \eqref{eq-obj}:
\begin{align}\label{eq-obj-constraint}
\bx^*= \underset{\left\{{x}_{i}\right\}_{i=1}^{n}}{{\arg\min}} & ~ \frac{1}{n}\sum_{i=1}^{n} f_{i}\left({x}_{i}\right) \\\nonumber
\text { s.t. } & ~ {x}_{i}=x_{j}, ~ \forall j\in \mathcal{N}_i, ~ \forall i.
\end{align}
Here, $\bx^*\triangleq [x^*;\ldots;x^*]\in \mathbb{R}^{nd}$ and $\mathcal{N}_i$ is the set of neighbors of node $i$. The equivalence between \eqref{eq-obj} and \eqref{eq-obj-constraint} holds when the network is connected.
Decentralized optimization problems in the form of \eqref{eq-obj-constraint} appear in various applications, such as deep learning \cite{liu2020decentralized}, sensor networking \cite{qin2020randomized}, statistical learning \cite{beznosikov2021distributed}, etc.

Decentralized algorithms for solving \eqref{eq-obj-constraint} are well studied. All nodes cooperatively obtain the common optimal solution $x^*$, simultaneously minimizing the objective function and reaching consensus. Generally speaking, minimization is realized by inexact descent on local objective functions and consensus is realized by variable averaging with a mixing matrix \cite{xiao2004fast}. Below, we briefly review the existing first-order and second-order decentralized algorithms for solving \eqref{eq-obj-constraint}.

\subsection{Decentralized First-order Methods}
First-order methods enjoy low per-iteration computational complexity and thus are popular. Decentralized gradient descent (DGD) is studied in \cite{nedic2009distributed,lian2018asynchronous}, where each node updates its local copy by a weighted average step on local copies from its neighbors, followed by a minimization step along its local gradient descent direction. With a fixed step size, DGD only converges to a neighborhood of $x^*$. This disadvantage can in part be explained by the observation that the local gradient is generally not a satisfactory estimate of the global one, even though the local copies are all equal to the optimal solution $x^*$. To construct a better local direction, various works with bias-correction techniques are proposed, such as primal-dual \cite{shi2015extra,srinivasa2019decentralized}, exact diffusion \cite{yuan2018exact1}, and gradient tracking \cite{li2020communication,koloskova2021improved}. For example, gradient tracking replaces the local gradient in DGD with a local gradient approximation obtained by the dynamic average consensus (DAC) technique, which leads to exact convergence with a fixed step size. Unified frameworks for first-order algorithms are investigated in \cite{alghunaim2020decentralized,xu2021distributed}.

In the centralized setting, it is well-known that convergence of first-order algorithms suffer from dependence on $\kappa_F$, the condition number of the objective function $F$. In the decentralized setting, the dependence is not only on $\kappa_F$ but also on the network. Specifically, let  $\sigma$ be the second largest singular value of the mixing matrix used in decentralized optimization and $\frac{1}{1-\sigma}$ be the condition number of the underlying communication graph. A network with larger $\frac{1}{1-\sigma}$ has a weaker information diffusion ability. For  strongly convex and smooth problems, the work \cite{scaman2017optimal} establishes the lower bounds 
$O\left(\sqrt{\kappa_{F}} \log \frac{1}{\epsilon}\right)$  and  $O\left(\sqrt{\frac{\kappa_{F}}{1-\sigma}} \log \frac{1}{\epsilon}\right)$ on the computation and communication costs for decentralized first-order algorithms to reach an $\epsilon$-optimal solution, respectively.
The lower bounds are  achieved or nearly achieved in  \cite{li2018sharp,ye2020multi}, where  multi-step consensus is introduced to balance the computation and communication costs.


\subsection{Decentralized Second-order Methods}
In the centralized setting, Newton's method is proved to have a locally quadratic convergence rate that is independent of $\kappa_F$. However, whether there is a communication-efficient decentralized variant of Newton's method with $\kappa_F$-independent  super-linear convergence rate under mild assumptions is still an open question. On the one hand, some decentralized second-order methods have provably faster rates but suffer from inexact convergence, high communication cost, or requiring strict assumptions. The work  \cite{mansoori2017superlinearly} extends the network Newton's method in \cite{mokhtari2016network} for minimizing a penalized approximation of \eqref{eq-obj} and shows that the convergence rate is super-linear in a specific neighborhood near the optimal solution of the penalized problem. Beyond this neighborhood, the rate becomes linear. The work \cite{tutunov2019distributed} proposes an approximate Newton's method for the dual problem of \eqref{eq-obj-constraint} and establishes a super-linear convergence rate within a neighborhood of the primal-dual optimal solution. However, in each iteration, it needs to solve the primal problem exactly to obtain the dual gradient and call a solver to obtain the local Newton direction. The work  \cite{zhang2020distributed}  proposes a decentralized adaptive Newton's method, which uses the communication-inefficient flooding technique to make  the global gradient and Hessian available to each node. In this way, each node conducts exactly the same update so that the global super-linear convergence rate of the centralized Newton's method with Polyak's adaptive step size still holds.
The work  \cite{di2016next} proposes a decentralized Newton-type method with cubic regularization and proves faster convergence up to statistical error under the assumption that each local Hessian is close enough to the global Hessian. The work  \cite{berglund2021distributed} studies
quadratic local objective functions and shows that for a distributed Newton's method,  the computation complexity depends only logarithmically on $\kappa_F$ with the help of exchanging the entire Hessian matrices. The algorithm in \cite{berglund2021distributed} is close to that in \cite{zanella2011newton}, but the latter has no  convergence rate guarantee.

On the other hand, some works are devoted to developing efficient decentralized second-order algorithms with similar computation and communication costs per iteration to  first-order algorithms. However, these methods only have globally linear convergence rate, which is no better than that of first-order methods \cite{mokhtari2016dqm,eisen2019primal,sun2019distributed,zhang2021newton,wei2021decentralized,zhang2022variance,zhang2022variance-2}. Here we summarize several reasons for the lack of provably faster rates: (i) The information fusion over the network is realized by averaging consensus, whose convergence rate is at most linear \cite{xiao2004fast}. (ii) The global Hessian is estimated just from local Hessians \cite{mokhtari2016dqm,eisen2019primal,sun2019distributed,zhang2021newton,wei2021decentralized} or from Hessian inverse approximations constructed with local gradient approximations \cite{zhang2022variance,zhang2022variance-2}. The purpose is to avoid the communication of entire Hessian matrices, but a downside is that the nodes are unable to fully utilize the global second-order information. (iii) The global Hessian matrices are typically assumed to be uniformly bounded, which simplifies the analysis but leads to under-utilization of the curvature information \cite{mokhtari2016dqm,eisen2019primal,sun2019distributed,zhang2021newton,wei2021decentralized,zhang2022variance,zhang2022variance-2}. (iv) For the centralized Newton's method, backtracking line search is vital for convergence analysis. It adaptively gives a small step size at the early stage to guarantee global convergence with arbitrary initialization and always gives a unit step size after reaching a neighborhood of the optimal solution to guarantee locally quadratic convergence rate. However, backtracking line search is not affordable in the decentralized setting since it is expensive for all the nodes to jointly calculate the global objective function value.

To the best of our knowledge, there is no decentralized Newton's method that, under mild assumptions, is not only  communication-efficient but also inherits the $\kappa_F$-independent  super-linear convergence rate of the centralized Newton's method. Therefore, in this paper we aim to address the following question: Can we design a communication-efficient decentralized Newton's method that has a provably $\kappa_F$-independent  super-linear convergence rate?

\subsection{Major Contributions}
To answer these questions, we  propose a decentralized Newton's method with multi-step consensus and compression and establish its convergence rate.  Roughly speaking, our method proceeds as follows. In each iteration, each node moves one step along a local approximated Newton direction, followed by variable averaging to improve consensus. To construct the local approximated Newton direction, we use the DAC technique to obtain a gradient approximation and a Hessian approximation, which track the global gradient and  global Hessian, respectively.
To avoid having each node to transmit the entire local Hessian approximation, we design a compression procedure with error compensation to estimate the global Hessian in a communication-efficient way. In other words, each node is able to obtain more accurate curvature information by exchanging the compressed local Hessian approximations with its neighbors, without incurring a high communication cost.
In addition, to balance between computation and communication costs, we use multi-step consensus for communicating the local copies of the decision variable and the local gradient approximations. Multi-step consensus helps to obtain not only a globally linear rate that is independent of the graph but also a faster local convergence rate.

Theoretically, we show, with a novel analysis, that our proposed method enjoys a provably faster convergence rate than those of decentralized first-order methods. The convergence process is split into two stages. In stage I, we use a small step size and get  globally linear convergence at the contraction rate of $1-O\left(\frac{1}{\kappa_F}\right) \min\left\{ \frac{(1-\sigma^2)^3}{\sigma^{m-1}}, \frac{1}{2}\right\}$ with arbitrary initialization. Here,  $\frac{1}{1-\sigma}$ is the condition number of the graph, $\kappa_{F}$ is the condition number of the objective function, and $m$ is the number of consensus inner steps. This globally linear rate holds for any $m\ge 1$. As a special case, when $m\ge \frac{\log2(1-\sigma^2)^3}{\log \sigma}+1$, the contraction rate in stage I becomes $1-O\left(\frac{1}{\kappa_F}\right)$, which is independent of the graph.  When the local copies are close enough to the optimal solution, the algorithm enters stage II, where we use a unit step size  and get the  faster  convergence  rate of $\sigma^{\frac{m}{2}}$.  This implies that we have a $\kappa_F$-independent linear rate when $m$ is a constant and an asymptotically super-linear rate when $m$ goes to infinity. No matter what $m$ is, the total communication complexity in stage II is $O\left(\frac{1}{-\log \sigma} \log \frac{1}{\epsilon}\right)$.
A comparison of the computational complexity of existing decentralized first-order and second-order methods are summarized in Table \ref{table-compare}.
\begin{table}[h]
	\centering
	\caption{Computational complexity to reach an $\epsilon$-optimal solution for decentralized consensus optimization algorithms}
	\label{table-compare}
	\begin{tabular}{c|c}
		\hline
		Algorithm   & Computational complexity   \\ \hline
		DLM \cite{ling2015dlm} &   $O\left(\max\left\{\frac{\kappa_F^2\lambda_{\max}(\mathcal{L}_u)}{\lambda_{\min}(\mathcal{L}_u)},\frac{(\lambda_{\max}(\mathcal{L}_u))^2}{\lambda_{\min}(\mathcal{L}_u)\hat{\lambda}_{\min}(\mathcal{L}_o)}\right\}\log\frac{1}{\epsilon}\right)\footnotemark{}$ \\ \hline
		EXTRA \cite{shi2015extra}  &    $O\left(\frac{\kappa_F^2}{1-\sigma}\log\frac{1}{\epsilon}\right)$ \footnotemark{}\\ \hline
		GT \cite{qu2017harnessing} &   $O\left(\frac{\kappa_F^2}{(1-\sigma)^2}\log\frac{1}{\epsilon}\right)$ \\ \hline
		DQM \cite{mokhtari2016dqm} &   $O\left(\max\left\{\left(\frac{\lambda_{\max}(\mathcal{L}_u)}{\hat{\lambda}_{\min}(\mathcal{L}_o)}\right)^2,  \frac{\kappa_F\lambda_{\max}(\mathcal{L}_u)}{\hat{\lambda}_{\min}(\mathcal{L}_o)} \right\} \log\frac{1}{\epsilon}\right)$ \footnotemark{} \\ \hline
		ESOM \cite{mokhtari2016decentralized} &  $O\left(  \frac{\kappa_F^2}{\hat{\lambda}_{\min}(I_n-W)}  \log\frac{1}{\epsilon}\right)$ \footnotemark{} \\ \hline
		NT  \cite{zhang2021newton}  & $O\left(\max\left\{ \kappa_F^2+\kappa_F \sqrt{\kappa_g}, \tfrac{\kappa_g^{3/2}}{\kappa_F} + \kappa_F\sqrt{\kappa_g} \right\} \log{\frac{1}{\epsilon}} \right)$ \footnotemark{}\\ \hline
		Stage I &$O\left(\kappa_F\max\left\{\frac{{\sigma^{m-1}}}{(1-\sigma^2)^3},2\right\} \log \frac{1}{\epsilon}\right)$\\\hline
		Stage II & $O\left(\frac{1}{-m \log\sigma}\log \frac{1}{\epsilon}\right)$ \footnotemark{}\\ \hline	
	\end{tabular}
\end{table}
\footnotetext[1]{Here, $\mathcal{L}_u$ and $\mathcal{L}_o$ are the unoriented and oriented Laplacian defined in \cite{ling2015dlm}, respectively. The rate is obtained when $\alpha=\frac{L_1\kappa_F}{\lambda_{\min}(\mathcal{L}_u)}$ and $\epsilon={L_1\kappa_F}$ with $L_1$ being the constant of Lipschitz continuous gradient.}
\footnotetext[2]{Here, $W$ is the mixing matrix,   $\tilde{W}=\frac{I_n+W}{2}$, and $\alpha=\frac{0.5\hat{\lambda}_{\min}(\tilde{W})}{L_1\kappa_F}$.}
\footnotetext[3]{Here, the convergence is local and $\alpha=\frac{L_1}{\lambda_{\max}(\mathcal{L}_u)\hat{\lambda}_{\min}(\mathcal{L}_0)}$.}
\footnotetext[4]{Here, the convergence is local and the number of consensus inner steps goes to infinity.}
\footnotetext[5]{Here, $\kappa_g=\frac{\lambda_{\max}(I_n-W)}{\hat{\lambda}_{\min}(I_n-W)}$ as defined in \cite{zhang2021newton} and the convergence is local.}
\footnotetext[6]{Here, $m$ is set as a constant. }

\textbf{Notation.}
We use
$I_d$ to denote the $d\times d$ identity matrix, $1_n$ to denote the $n$-dimensional column vector of all ones, $\|\cdot\|$  to denote the Euclidean norm of a vector or the largest singular value of a matrix,  $\|\cdot\|_F$ to denote the Frobenius norm, and  $\otimes$ to denote the Kronecker product.
For a matrix $A$, we use $A\ge 0$ to indicate that each entry of $A$ is non-negative. For a symmetric matrix $A$, we use $A\succeq 0$ and $A\succ0$ to indicate that $A$ is
positive semidefinite and positive definite, respectively. For two matrices $A$ and $B$ of the same dimensions, we use $A\ge B$,  $A\succeq B$, and $ A\succ B$ to indicate that $A-B\ge 0$,  $A-B \succeq 0$, and $A-B \succ 0$, respectively.
 We use  $\lambda_{\max}(\cdot)$, $\lambda_{\min}(\cdot)$, and $\hat{\lambda}_{\min}$ to denote the largest, smallest, and the smallest positive eigenvalues of a matrix, respectively.

For $x_1, \ldots, x_n \in \mathbb{R}^d$, we define the aggregated   variable  $\bx=[x_1;\ldots;x_n]\in \mathbb{R}^{nd}$. The  aggregated variables $\bd$ and $\bg$
are defined similarly.
We define the average
variable over all the nodes at time step $k$ as
$\ox^k=\frac{1}{n}\sum_{i=1}^n x_i^k\in\mathbb{R}^d$.
The average variables $\od^k$ and $\og^k$ are defined similarly.
We define the aggregated gradient at time step $k$ as $\nabla f(\bx^k)=\left[\nabla f_1(x_1^k);\ldots; \nabla f_n(x_n^k)\right]\in \mathbb{R}^{nd}$, the  average of all the local gradients at the local variables as $\overline{\nabla f }(\bx^k)=\frac{1}{n} \sum_{i=1}^{n} \nabla f_i(x_i^k) \in \mathbb{R}^d$, and the average of all
the local gradients at the common average $\ox^k$ as $\nabla F(\ox^k)=\frac{1}{n} \sum_{i=1}^{n} \nabla f_i(\ox^k)\in \mathbb{R}^d$. The aggregated Hessian $\nabla^2 f(\bx^k)\in \mathbb{R}^{nd\times
d}$, the average of all the local Hessian at the local variables $\overline{\nabla^2 f}(\bx^k)\in \mathbb{R}^{d\times d}$, and the average of all the local Hessians at the common average $\nabla^2 F(\ox^k)\in \mathbb{R}^{d\times d}$  are defined similarly.
Given the matrices $H_i^k \in \mathbb{R}^{d \times d}$, we define the aggregated matrix   $\bH^{k}=[H_1^{k};\ldots;H_n^{k}]\in \mathbb{R}^{nd\times d}$.  The aggregated matrices $\bE^k$, $\tilde{\bH}^k$, and $\hat{\bH}^k$ are defined similarly. We define the average variable over all the nodes at time step $k$ as $\oH^k=\frac{1}{n}\sum_{i=1}^{n} H_i^k$ and  $\diag\{H_i\}\in \mathbb{R}^{nd \times nd}$ as  the block diagonal matrix whose $i$-th block is $H_i\in \mathbb{R}^{d\times d}$.  We define
$\bW=W\otimes I_d\in\mathbb{R}^{nd \times nd}$ and
{$\bW^{\infty}=\frac{1_n 1_n^T}{n}\otimes I_d\in\mathbb{R}^{nd \times nd}$}.

\section{Problem Setting and Algorithm Development}
In this section, we give the problem setting  and the basic assumptions. Then,
we propose a decentralized Newton's method with multi-step consensus and compression.
\subsection{Problem Setting}
We consider an undirected, connected graph  $\mathcal{G}=(\mathcal{V}, \mathcal{E})$, where $\mathcal{V}=\{1,\ldots,n\}$ is the set of nodes and $\mathcal{E}\subseteq \mathcal{V}\times \mathcal{V}$ is the set of edges. We use $(i,j)\in \mathcal{E}$ to indicate that nodes $i$ and $j$ are neighbors, and neighbors are allowed to communicate with each other. We use $\mathcal{N}_i$ to denote the set of neighbors of node $i$ and itself. We introduce a mixing matrix $W \in \mathbb{R}^{n \times n}$ to model the communication among nodes. The mixing matrix is assumed to satisfy the following:
\begin{asp}\label{asm-W}
	The mixing matrix $W$ is  non-negative, symmetric, and doubly stochastic (i.e., $w_{ij} \ge 0$ for all $i, j \in \{1,\ldots,n\}$, $W=W^T$, and $W1_n = 1_n$) with $w_{ij} = 0$ if and only if $j \notin $ $\mathcal{N}_{i}$.
\end{asp}

Assumption \ref{asm-W} implies that the null space of $I_n - W$ is $\mbox{span}(1_n)$, the eigenvalues of $W$ lie in $(-1,1]$, and  1 is an eigenvalue of $W$ of multiplicity 1.
Let $\sigma$ be the second largest singular value of $W$. Under Assumption \ref{asm-W}, we have
\begin{equation*}
	\sigma=\|W - W^\infty\| <1.
\end{equation*}
Usually, $\sigma$ is used to represent the connectedness of the graph \cite{nedic2009distributed,shi2015extra}.
Mixing matrices satisfying Assumption \ref{asm-W} are frequently used in decentralized optimization over an undirected, connected network; see, e.g., \cite{boyd2004fastest} for details.

Throughout the paper, we make the following assumptions on the local objective functions.
\begin{asp}\label{asp:lipschitz-continuous}
Each $f_i$ is twice-continuously differentiable. Both the gradient and Hessian are Lipschitz continuous, i.e.,
\begin{align}
\| \nabla f_i(x) - \nabla f_i(y) \| \leq L_1 \|x - y\|
\end{align}
and
\begin{align}
\| \nabla^2 f_i(x) - \nabla^2 f_i(y) \| \leq L_2 \|x - y\|
\end{align}
for all $x,y \in \mathbb{R}^d$, where $L_1>0$ and $L_2>0$ are the Lipschitz constants of the local gradient and local Hessian, respectively.
\end{asp}
\begin{asp}\label{ass-strongly-cvx}
The entire objective $F$ is $\mu$-strongly convex for some constant $\mu>0$, i.e.,
\begin{align}
\nabla^2 F(x) \succeq \mu I_d
\end{align}
for all $x\in \mathbb{R}^d$, where $\mu$ is the strong convexity constant.
\end{asp}

We should remark that in Assumption \ref{ass-strongly-cvx}, we only assume the entire objective function $F$ to be strongly convex. The local objective function $f_i$ on each node could even be nonconvex, which makes our analysis more general.

To avoid having each node to communicate the entire local Hessian approximation, we design a compression procedure with a deterministic contractive compression operator  $\mathcal{Q}(\cdot)$ that satisfies the following assumption.
\begin{asp}\label{asp:compression}
The deterministic contractive compression operator $\mathcal{Q}: \mathbb{R}^{d \times d} \rightarrow \mathbb{R}^{d \times d}$ satisfies
\begin{align}\label{eq-contractive}
\|\mathcal{Q} (A) - A\|_F \leq (1 - \delta) \|A\|_F
\end{align}
for all $A\in \mathbb{R}^{d\times d}$, where
 $\delta \in (0, 1]$ is a constant determined by the compression operator.
\end{asp}

We now present two concrete examples of such an operator.
Let $A=\sum_{i=1}^{d}\sigma_i u_i v_i^T$ be the singular value decomposition of the matrix $A$, where $\sigma_i$ is the $i$-th largest singular value with $u_i$ and $v_i$ being the corresponding singular vectors. The Rank-$K$ compression operator outputs $\mathcal{Q}(A)=\sum_{i=1}^{K} \sigma_i u_i v_i^T $, which is a rank-$K$ approximation of $A$. For Top-$K$ compression operator, $\mathcal{Q}(A)$ keeps the $K$ largest entries (in terms of the absolute value) of the matrix $A$ and sets the other entries as zero. For more details of compression operators, one can refer to \cite{safaryan2021fednl,islamov2021distributed}.

The following proposition shows that both the Rank-$K$ and Top-$K$ compression operators satisfy Assumption \ref{asp:compression}.

\begin{props}\label{prop-compression operator}
For the Rank-$K$ and Top-$K$ compression operators, Assumption \ref{asp:compression} holds with $\delta=\frac{K}{2d}$ and $\delta=\frac{K}{2d^2}$, respectively.
\end{props}
\begin{proof}
See Appendix \ref{app-compression operator}.
\end{proof}

\begin{rmk}
Different from the random compression operators used in first-order algorithms  \cite{richtarik2021ef21,qian2021basis}, we use  deterministic compression operators. This is because any realization not satisfying \eqref{eq-contractive} may lead to a non-positive semidefinite Hessian approximation and thus leads to the failure of the proposed Newton's method.
\end{rmk}

\subsection{Algorithm Development}
In this section, we propose a decentralized Newton's method with multi-step consensus and compression. In iteration $k$, node $i$ first conducts one minimization step along a local approximated Newton direction $d_i$ and then communicates the result with its neighbors for $m$ rounds to compute
\begin{align}\label{eq-alg-x}
x_{i}^{k+1}=\sum_{j\in \mathcal{N}_i} (W^m)_{i j} \left( x_{j}^{k}-\alpha d_{j}^{k} \right).
\end{align}
Here, $\alpha>0$ is a step size and $(W^m)_{i j}$ is  the $(i,j)$-th  entry of $W^m$. Such multi-step consensus costs $m$ rounds of communication. As we will show in the next section, the multi-step consensus balances between computation and communication and is vital to get a provably fast convergence rate.

To update the local direction,
we use the DAC technique to obtain a gradient approximation and a Hessian approximation, which track the global gradient and global Hessian, respectively. The gradient approximation $g_i^{k+1}$ on node $i$ is computed by
\begin{align}\label{eq-alg-g}
g_{i}^{k+1}=\sum_{j\in \mathcal{N}_i} (W^m)_{i j} \left( g_{j}^{k}+\nabla f_j(x_j^{k+1})-\nabla f_j(x_j^k)\right)
\end{align}
with initialization $g_i^0=\nabla f_i(x_i^0)$.
Similar to \eqref{eq-alg-x}, we use multi-step consensus to make $g_i^{k+1}$  a more accurate gradient approximation.

\begin{algorithm}[h]
	\caption{Decentralized Newton's method}
	\label{alg-1}
	\begin{algorithmic}
		\STATE $\textbf{Input:}$ $\bx^0, \bd^0, g_i^0=\nabla f_i(x_i^0), H_i^0=\nabla^2 f_i(x_i^0)$, $E_i^0$, $\tilde{H}_i^0$,  $\alpha$,  $\gamma$, $m$, $M$.
		\FOR {$k = 0, 1, 2, \ldots$ }
		\STATE $\bx^{k + 1} = \bW^m  (\bx^k - \alpha \bd^k)$
		\STATE $\bg^{k + 1} = \bW^m  \left(\bg^k + \nabla f(\bx^{k + 1}) - \nabla f(\bx^{k}) \right)$\\
		\rule[0.1\baselineskip]{0.42\textwidth}{0.1pt}
		\STATE \bf Compression Procedure\\
		\STATE	$\bE^{k + 1} = \bE^k + \bH^k -\tilde{\bH}^k - \mathcal{Q} (\bE^k + \bH^k -\tilde{\bH}^k)$
		\STATE	$\tilde{\bH}^{k + 1} = \tilde{\bH}^k +  \mathcal{Q} (\bH^k -\tilde{\bH}^k)$
		\STATE$\hat{\bH}^{k} = \tilde{\bH}^k + \mathcal{Q} (\bE^k + \bH^k -\tilde{\bH}^k)$
		\STATE $\bH^{k + 1} = \bH^k - \gamma (I_{nd} - \bW) \hat{\bH}^{k} \!+\! \nabla^2 f(\bx^{k + 1})\! - \!\nabla^2 f(\bx^k)$\\
		\rule[0.1\baselineskip]{0.42\textwidth}{0.1pt}
		\STATE $ \bd^{k + 1} \approx \left( \diag\{H_i^{k + 1}\} +M I_{nd} \right)^{-1} \bg^{k + 1}$
		\ENDFOR
	\end{algorithmic}
\end{algorithm}

To obtain the Hessian approximation, we also  use DAC to mix the second-order curvature information over the network but keep in mind that communicating the entire local Hessian approximation leads to a high communication cost. Thus, we design a compression procedure with error compensation to estimate the global Hessian in a communication-efficient way. In other words, each node is able to obtain more accurate global curvature information by exchanging the compressed local Hessian approximation with its neighbors, without incurring a high communication cost. The Hessian approximation $H_i^{k+1}$ on node $i$ is given by
\begin{align}\label{eq-alg-H}
&h_i^{k + 1} = h_i^k +  \mathcal{Q} (H_i^k - h_i^k),\nonumber\\
&  E_i^{k + 1} = E_i^k + H_i^k - h_i^k - \mathcal{Q} (E_i^k + H_i^k - h_i^k),\\
&\hat{H}_i^{k} = h_i^k + \mathcal{Q} (E_i^k + H_i^k - h_i^k),\nonumber\\
&H_i^{k + 1} = H_i^k \!-\! \gamma \sum_{j\in \mathcal{N}_i}w_{ij}( \hat{H}_i^k - \hat{H}_j^k) \! +\! \nabla^2 f_i(x_i^{k + 1}) \!-\!\nabla^2 f_i(x_i^k)\nonumber
\end{align}
{with initialization $H_i^0=\nabla^2 f_i(x_i^0)$, where $\gamma>0$ is a parameter.
Compared with DAC without compression, i.e., $H_i^{k + 1} = H_i^k - \gamma \sum_{j\in \mathcal{N}_i}w_{ij}( H_i^k - H_j^k)  + \nabla^2 f_i(x_i^{k + 1}) - \nabla^2 f_i(x_i^k)$, the term $w_{ij}( H_i^k - H_j^k) $ is replaced by $ w_{ij}( \hat{H}_i^k - \hat{H}_j^k)$, which can be constructed with compressed communication.
There are two techniques to compensate for the compression error in the construction of $\hat{H}_i^k$: (i) We introduce $\tilde{H}_i^k$ as a counterpart of $H_i^k$ and compress their difference $H_i^k-\tilde{H}_i^k$. (ii) We add $E_i^{k}$---the compression error in the $(k-1)$-st iteration---back into the difference $H_i^k-\tilde{H}_i^k$ in the $k$-th iteration for error feedback and compress $E_i^k+H_i^k-\tilde{H}_i^k$. Intuitively, there is no compression error when the algorithm converges so that $H_i^k\to \tilde{H}_i^k$ and $E_i^k\to 0$. This intuition will be verified by our analysis later (see Proposition \ref{prop-2}). It is worth noting that we only use one round of communication per iteration to construct the Hessian approximation.

With the local gradient and Hessian approximations, we are ready to update the local direction. To avoid calculating the inverse of the local Hessian approximation, we
utilize an early-terminated conjugate gradient (CG) method \cite{dai1999nonlinear} to obtain a local direction $d_i^{k+1}$ via
\begin{align}\label{eq-alg-d}
\left(H_i^{k+1}+MI_d\right)d_i^{k+1}\approx g_i^{k+1},
\end{align}
where $M >0$ is a regularization parameter. The accuracy of the CG step will be given later (see Fact \ref{fact-CG}). The proposed algorithm can be written in a compact form as summarized in Algorithm \ref{alg-1}. With a slight abuse of notation in the compact form, given an aggregated matrix $A=[A_1;\ldots;A_n]\in\mathbb{R}^{nd\times d}$, we use $\mathcal{Q}(\cdot)$ to denote the block-wise compression operator such that $\mathcal{Q}(A)=[\mathcal{Q}(A_1);\ldots;\mathcal{Q}(A_n)]\in\mathbb{R}^{nd \times d}$.

In Algorithm \ref{alg-1}, computing
$\bW\hat{\bH}^k$ requires communicating the uncompressed matrices $\hat{H}_1^k, \ldots, \hat{H}_n^k $, which is not communication-efficient.  As shown in \cite{liao2021compressed}, there is an equivalent but communication-efficient implementation of the compression procedure, summarized in Algorithm \ref{alg-2}. The basic idea is to introduce an auxiliary variable $\tilde{\bH}_w^k$ that is equal to $\bW\tilde{\bH}^k$ and use it to construct $\hat{\bH}_w^k$ that is equal to $\bW\hat{\bH}^k$.  Algorithm \ref{alg-2} is communication-efficient since the nodes only communicate the compressed variables $\bQ^k$ and $\hat{\bQ}^k$. For simplicity, we study Algorithm \ref{alg-1} in our convergence analysis.

\begin{algorithm}[h]
\caption{Communication-efficient implementation}
\label{alg-2}
\begin{algorithmic}
\STATE {$\textbf{Input:}$ $\bx^0, \bd^0, g_i^0=\nabla f_i(x_i^0), H_i^0=\nabla^2 f_i(x_i^0)$, $\tilde{\bH}^0$, $\tilde{\bH}_w^0=\bW\tilde{\bH}^0$, $\alpha$,  $\gamma$, $m$, $M$. }
\FOR{$k = 0, 1, 2, \ldots$ }
\STATE 	$\bx^{k + 1} = \bW^m  (\bx^k - \alpha \bd^k)$
\STATE	$\bg^{k + 1} = \bW^m  \left(\bg^k + \nabla f(\bx^{k + 1}) - \nabla f(\bx^{k}) \right)$
\rule[0.1\baselineskip]{0.42\textwidth}{0.1pt}
\STATE \bf Compression Procedure\\
\STATE 	${\bQ}^k=\mathcal{Q} (\bH^k -\tilde{\bH}^k)$
\STATE 	$\hat{\bQ}^k=\mathcal{Q} (\bE^k + \bH^k -\tilde{\bH}^k)$
\STATE  $\tilde{\bH}^{k + 1} = \tilde{\bH}^k +  \bQ^k$
\STATE 	$\tilde{\bH}_w^{k + 1} = \tilde{\bH}_w^k +  \bW \bQ^k$
\STATE 	$\hat{\bH}^{k} = \tilde{\bH}^k + \hat{\bQ}^k$
\STATE 	$\hat{\bH}_w^{k} = \tilde{\bH}_w^k + \bW\hat{\bQ}^k$
\STATE 	$\bE^{k + 1} = \bE^k + \bH^k -\tilde{\bH}^k - \hat{\bQ}^k$
\STATE 	$\bH^{k + 1} = \bH^k - \gamma (\hat{\bH}^k - \hat{\bH}_w^k)  + \nabla^2 f(\bx^{k + 1}) \!-\! \nabla^2 f(\bx^k)$
\rule[0.1\baselineskip]{0.42\textwidth}{0.1pt}
\STATE $ \bd^{k + 1} \approx \left( \diag\{H_i^{k + 1}\} +M I_{nd} \right)^{-1} \bg^{k + 1}$
\ENDFOR
\end{algorithmic}
\end{algorithm}

\section{Convergence Analysis}\label{sec-analysis}
In this section, we conduct a novel two-stage analysis of our proposed Algorithm \ref{alg-1} and establish its convergence rate. Our analysis reveals that Algorithm \ref{alg-1} is provably faster than the first-order algorithms. For the centralized Newton's method, to get a globally linear convergence rate with arbitrary initialization and a locally quadratic convergence rate, one often resorts to backtracking line search, which adaptively gives a small step size at the early stage and always gives a unit step size within a neighborhood of the optimal solution \cite{nocedal2006numerical}. However, backtracking line search becomes expensive in the decentralized setting. Nevertheless, we can mimic the process of backtracking line search and split the convergence process into two stages: The algorithm uses a small step size in stage I and converges linearly until the local copies are close enough to the optimal solution. Then, the algorithm enters stage II and starts to use a unit step size; we will show a local faster-than-linear rate by taking advantage of the curvature information which is not exploited in stage I.

Before starting the analysis, we specify the accuracy of the CG method with the following fact \cite{dai1999nonlinear}.
\begin{fct}\label{fact-CG}
With at most $d$ iterations for each node, the CG step yields
\begin{align}\label{ineq:inexact_inverse}
\left( \diag\{H_i^{k + 1}\}+M I_{nd} \right) \bd^{k + 1} = \bg^{k + 1} + \br^{k + 1}
\end{align}
with
$$\| \br^{k + 1}\| \leq c_k \|\bg^{k + 1}\|$$
for any $0\le c_k\le 1$.
\end{fct}

\subsection{Stage I: Globally Linear Convergence}
The convergence analysis of stage I is inspired by the work of \cite{zhang2022variance}, where a general framework of stochastic decentralized quasi-Newton methods is proposed. A globally linear convergence rate is established under the assumption that the constructed Hessian inverse approximations are positive definite with bounded eigenvalues. Similar to \cite{zhang2022variance}, we define two constants $M_1$ and $M_2$ as
\begin{align}\label{def:M1M2}
	M_1 \triangleq \mu \!+\! M - L_2 \sqrt{\frac{u_1^0}{n}} - \tilde{u}_2^0, ~ M_2 \triangleq L_1 \!+\! M + L_2 \sqrt{\frac{u_1^0}{n}}\!+\! \tilde{u}_2^0,
\end{align}
where $u_1^0$ is defined in \eqref{def:u1k} and $\tilde{u}_2^0$ is a constant given in \eqref{eq-tildeu20}. When choosing the parameter $M \geq  \tilde{u}_2^0 + L_2 \sqrt{\frac{u_1^0}{n}}$, we have $M_2 \geq M_1 >0$. We establish the globally linear convergence in stage I under  the condition
\begin{align}\label{eq-M1M2}
	M_1 I_d \preceq H_i^{k} +M I_d \preceq M_2 I_d, ~\forall i\in \{1,\ldots,n\}.
\end{align}
We will prove that the sequence $\{H_i^k\}_{k\ge 0}$ generated by Algorithm \ref{alg-1} satisfy condition \eqref{eq-M1M2} for all $k\ge 0$ (see Proposition \ref{prop-3}).

\subsubsection{Main Theorem for Stage I}

In stage I, we want to establish the globally linear convergence of Algorithm \ref{alg-1} with arbitrary initialization. To this end, it is sufficient to only use the uniform bound instead of the curvature information of Hessian approximations  given in \eqref{eq-M1M2}. This significantly simplifies our analysis of stage I. To begin, let us define
\begin{equation*}
\bq_1^k\triangleq
\begin{pmatrix}
	\|\bx^{k} - \bW^\infty \bx^{k} \|^2 \\
	\frac{1}{L_1^2} \|\bg^{k} - \bW^\infty \bg^{k} \|^2 \\
	\frac{n}{L_1} \left(F(\ox^{k}) - F(x^*)\right)
\end{pmatrix}
\end{equation*}
and
\begin{align}\label{def:u1k}
	u_1^k &\triangleq \left(1, \frac{(1 - \sigma^2)^2}{50}, 2\sigma^{m-1} \right) \bq_1^k,
\end{align}
where $u_1^k=0$ implies that $x_1^k=\cdots =x_n^k=x^*$ due to the strong convexity of $F$.
\begin{thm}\label{theo-1}
Under Assumptions  \ref{asm-W}--\ref{asp:compression}, if the parameters satisfy
\begin{align}\label{eq-a-c}
&M \geq L_2 \sqrt{\frac{u_1^0}{n}} + \tilde{u}_2^0, \quad \alpha \leq \min\left\{ \frac{M_1^2 (1 - \sigma^2)^3}{100 L_1 M_2{\sigma^{m-1}}},\frac{M_1^2}{200L_1M_2}\right\},\nonumber\\
 &c_k \leq \frac{M_1}{4M_2 \sqrt{2 \kappa_F}}, \quad
\gamma \leq \frac{ \delta^2 (1 - \sigma)}{50},
\end{align}
then for any  $m \geq 1$ and any $\bx^0$,
we have
\begin{align}\label{eq-rate-u}
u_1^{k+1} \leq \left(1 - \frac{\mu \alpha}{2 M_2} \right) u_1^k, \quad \forall k \geq 0.
\end{align}
\end{thm}

Theorem \ref{theo-1} implies that if the parameter $M$ is sufficiently large, the step sizes $\alpha$ and $\gamma$ are sufficiently small, and the CG step is sufficiently accurate, then Algorithm \ref{alg-1} converges linearly with any number of communication rounds $m$ and any initialization $\bx^0$.
\begin{rmk}
	By substituting \eqref{eq-a-c} into \eqref{eq-rate-u}, we have that the total number of iterations for Algorithm \ref{alg-1} to get an $\epsilon$-optimal solution is $$O\left(\kappa_F \max\left\{\frac{{\sigma^{m-1}}}{(1-\sigma^2)^3},2\right\} \log \frac{1}{\epsilon}\right),$$
	where we use $\tfrac{M_2}{M_1}=O(1)$ by setting $M\gg L_1 + \tilde{u}_2^0$.
	 If we set $m\ge \frac{\log 2(1-\sigma^2)^3}{\log \sigma}+1$, then the computational complexity becomes $$O\left( \kappa_{F}\log\frac{1}{\epsilon}\right),$$ which is independent on the graph. 	 The computational complexity of stage I is still $\kappa_F$-dependent because we only use the uniform bounds on the Hessian approximations and have not yet employed the curvature information. Nevertheless, this rate is still favorable since the goal of stage I is to guarantee globally linear convergence with arbitrary initialization. We will show a faster theoretical rate for stage II.
\end{rmk}
\begin{rmk}	
	Compared with the analysis of stochastic quasi-Newton methods in  \cite{zhang2022variance}, we consider the deterministic case and need novel techniques to control the inexactness caused by the CG step. Besides, we get a better theoretical computational complexity than that of \cite{zhang2022variance} by using DAC to mix local Hessian approximations.

\end{rmk}

\subsubsection{One-step Descent in Stage I}
Given that condition \eqref{eq-M1M2} holds at a certain time step $k_0$, the following proposition establishes one-step descent from $u_1^{k_0}$ to $u_1^{k_0+1}$.
\begin{props}\label{prop-1}
	Under Assumptions  \ref{asm-W}--\ref{asp:compression}, if $c_k\le 1$ and \eqref{eq-M1M2} holds at a certain time step $k_0$, then we have
\begin{align}\label{eq-rate-1}
	\bq_1^{k_0+1}\le \bJ^{[1]} \bq_1^{k_0}
\end{align}
with
\begin{equation*}
	\hspace{-2mm}{\bJ^{[1]}} \!\triangleq\!
	\begin{bmatrix}
		1\! -\! 0.49(1\! -\!\sigma^2) &
		{ 0.004 (1 \!-\! \sigma^2)^3}&
		\frac{{ 64} {\sigma^{2m}} \alpha^2 L_1^2}{(1 - \sigma^2) M_1^2}\\
		\frac{33}{1 - \sigma^2}& 1 - 0.49(1 \!- \!\sigma^2)& \frac{{ 64}{\sigma^{2m}}\alpha^2 L_1^2}{(1 - \sigma^2) M_1^2}\\
		\frac{(1 - \sigma^2)^3}{{20\sigma^{m-1}}} & \frac{(1 - \sigma^2)^3}{{80\sigma^{m-1}}} & 1 - \frac{\mu \alpha}{M_2}
	\end{bmatrix}.
\end{equation*}
Further, if the parameters $M$, $\alpha$, $c_k$, and $\gamma$ satisfy \eqref{eq-a-c}, then  \eqref{eq-rate-1} implies that
\begin{align}\label{eq-rate-u0}
	u_1^{k_0 + 1} \leq \left(1 - \frac{\mu \alpha}{2 M_2} \right) u_1^{k_0}.
\end{align}
\end{props}
\begin{proof}
	See Appendix \ref{prf-theorem1}.
\end{proof}

\subsubsection{Proof of Main Theorem for Stage I}
Thanks to Proposition \ref{prop-1}, to prove Theorem \ref{theo-1}, we only need to show that \eqref{eq-M1M2} holds for all $k\geq 0$. Observing that \eqref{eq-M1M2} gives bounds on the Hessian approximations, we need to take the specific compression procedure into consideration and give the convergence rate of the Hessian tracking error $\|\bH^{k + 1}-\bW^\infty \bH^{k + 1} \|_F$. To do this, we establish the following proposition to bound the compression error $\|\bE^{k + 1} \|_F$, the Hessian approximation difference $\|\bH^{k + 1} - \tilde{\bH}^{k + 1} \|_F$, and the Hessian tracking error $\|\bH^{k + 1} - \bW^\infty \bH^{k + 1} \|_F$. Let us define
\begin{equation*}
\bq_2^k\triangleq \begin{pmatrix}
		\|\bE^k\|_F \\
		\|\bH^k -\tilde{\bH}^k\|_F\\
		\|\bH^k - \bW^\infty \bH^k \|_F
	\end{pmatrix}
\end{equation*}
and
\begin{align*}
	u_2^k &\triangleq \left(\frac{\delta (1 - \sigma)}{8(1 - \delta)}, \frac{1 - \sigma}{4}, 1 \right) \bq_2^k.
\end{align*}
Here, $u_2^k=0$ implies that $\bE^k=0$, $\tilde{\bH}^k=\bH^k$, and $H_1^k=\cdots=H_n^k=\frac{1}{n}\sum_{i=1}^{n} \nabla f_i(x_i^k)$, where we use $\oH^{k+1}=\overline{\nabla^2 f}(\bx^{k+1})$ for all $k\ge 0$.
\begin{props}\label{prop-2}
	Under Assumptions \ref{asm-W},  \ref{asp:lipschitz-continuous},  and \ref{asp:compression}, if $\gamma \leq 1$, then for all $k\ge 0$, we have
	\begin{align}\label{eq-compress-linear}
		\bq_2^{k+1}\le \bJ^{[2]} \bq_2^k+L_2\|\bx^{k + 1}-\bx^k\|\begin{bmatrix}
			0\\ 1\\ 1
		\end{bmatrix}
	\end{align}
with
	\begin{equation*}
		\bJ^{[2]}\triangleq
		\begin{bmatrix}
			1 - \delta & 1 - \delta &  0&\\
			4 \gamma & \left(1 - \delta + 2 \gamma (1 - \delta) \right) & 2\gamma \\
			4 \gamma & 2\gamma (1 - \delta) & 1 - \gamma (1 - \sigma)
		\end{bmatrix}.
	\end{equation*}
Further, under Assumption \ref{ass-strongly-cvx}, if the parameters $M$, $\alpha$, $c_k$, and $\gamma$ satisfy \eqref{eq-a-c} and \eqref{eq-M1M2} holds at a certain time step $k_0$, then \eqref{eq-compress-linear} implies that
\begin{align}\label{eq-linear-v}
	u_2^{k_0 + 1} \leq \left(1 - \frac{\gamma}{2} (1 - \sigma) \right) u_2^{k_0} + \frac{15L_2}{4}  \sqrt{\sigma^{-(m-1)}u_1^{k_0}}.
\end{align}	
\end{props}
\begin{proof}
	See Appendix \ref{prf-theom-2}.
\end{proof}

Observing that Propositions \ref{prop-1} and \ref{prop-2} hold for a certain $k_0$, we show that \eqref{eq-M1M2} holds for all $k\geq 0$ via mathematical induction.
\begin{props}\label{prop-3}
Under the setting of Theorem \ref{theo-1}, considering the sequence $\{H_i^k\}_{k\ge 0}$ for any $i\in \{1,\ldots,n\}$ generated by Algorithm \ref{alg-1}, we have that condition \eqref{eq-M1M2} holds for all $k\geq 0$.
\end{props}
\begin{proof}
See Appendix \ref{sec-prof-assump-M}.
\end{proof}

Thus, combining  Propositions \ref{prop-1} and \ref{prop-3} directly gives \eqref{eq-rate-u}. This completes the proof of Theorem \ref{theo-1}.

\subsection{Stage II: Faster Local Convergence}
After stage I, all the local copies $x_1^k, \ldots, x_n^k$ are close enough to $x^*$ according to Theorem \ref{theo-1} and all the local Hessian approximations $H_1^k, \ldots, H_n^k$ are almost consensual according to Proposition \ref{prop-2}, as long as $k$ is sufficiently large. We will specify the number of iterations needed by stage I later (see \eqref{eq-k-stageI}). In stage I, we only use the uniform boundedness of the Hessian approximations and do not take advantage of the curvature information adequately. However, in stage II, to get a locally faster rate,  we need to bound the error between each local Hessian approximation $H_i^k$ and the global Hessian $\nabla^2 F(\ox^k)$ (see \eqref{eq-44}). After this, we further bound the error between local directions $d_i^k$ and the global Newton direction $(\nabla^2 F(\ox^{k} ))^{-1} \nabla F(\ox^{k})$ (see \eqref{eq-H-nabla2-1}). In this way, we can utilize the locally quadratic convergence rate of the centralized Newton's method to bound $\|\ox^{k} - x^*\|$ (see Corollary \ref{lem-xopt}). The analysis is novel compared with those of existing first-order and second-order methods and is vital to relate the decentralized Newton's method with the centralized one.

Let the proposed algorithm enter stage II after $K$ iterations with
\begin{align}\label{eq-k-stageI}
\hspace{-1em}	K \!\geq\! \frac{\frac{m}{2}\log \sigma - \log \frac{41\kappa_F}{\mu\sqrt{n}}  \left(\tilde{u}_2^0  + \frac{52L_2\kappa_F  \sqrt{\kappa_F}\sigma^{-\frac{5m}{4}}}{(1-\sigma^{m/2})(1 - \sigma^2)} \sqrt{u_1^0}\right)}{\log \phi}.
\end{align}
Let $M_1 \triangleq \frac{40\mu}{41}$ and $M_2 \triangleq L_1 + \frac{\mu}{41}$ in stage II, which are different from those in stage I but we use the same notation for simplicity. We establish the faster local rate in stage II under the condition
\begin{align}\label{ineq-M1M2}
	M_1 I_d \preceq H_i^{k + 1} \preceq M_2 I_d, \quad \forall\, i \in \{1,\ldots,n\}
\end{align}
for all $k\ge K$. Proposition \ref{prop-5} below shows that the sequence $\{H_i^k\}_{k\ge 0}$ for any $i\in \{1,\ldots,n\}$ generated by Algorithm \ref{alg-1} satisfies \eqref{ineq-M1M2} .

\subsubsection{Main Theorem for Stage II}
Define $\bq_3^k$ and $\bJ^{[3]}$ as in \eqref{eq-J3} and  $$
u_3^k\triangleq \left(1, \sigma^{-\frac{m}{4}}, 0.5 \sigma^{-\frac{3m}{4}}\right)\bq_3^k.
$$
We establish a faster local rate for stage II.
\begin{thm}\label{theo-3}
Under Assumptions \ref{asm-W}--\ref{asp:compression},
if the parameters
satisfy
\begin{align}\label{ineq-cond}
\hspace{-3mm}\alpha \!= \!1,~ M=0, ~ m > \frac{ 4 \log (4\kappa_F)}{ -\log \sigma}, c_k\le \frac{M_1\sigma^{m/2}}{40\mu \kappa_F}, \gamma \leq 1,
\end{align}
then for all $k\geq K$, we have
\begin{align}
u_3^{k+1} \leq \sigma^{\frac{m}{2}}u_3^k.
\end{align}
\end{thm}

Theorem \ref{theo-3} implies that we get a $\kappa_F$-independent linear rate when $m$ is a constant and an asymptotically super-linear rate when $m$ goes to infinity.
\begin{rmk}
	In stage II, we artificially set a unit step size $\alpha$ to establish a faster local rate. This is done by taking advantage of the curvature information. Note that existing techniques for adaptively choosing step sizes, such as backtracking line search, require evaluations of the entire objective function for many times per iteration. In the numerical experiments, we show that geometrically increasing the step size to unit works well.
\end{rmk}
\subsubsection{One-step Descent in Stage II}
To prove Theorem \ref{theo-3}, given  that condition \eqref{ineq-M1M2} holds for a certain time step $\tilde{k}_0$ with $\tilde{k}_0\ge K$, we establish one-step descent from $u_2^{\tilde{k}_0}$ to $u_2^{\tilde{k}_0+1}$.
\begin{props}\label{prop-4}
	Under Assumptions \ref{asm-W}--\ref{ass-strongly-cvx}, if $M=0$ and \eqref{ineq-M1M2} holds for a certain $\tilde{k}_0$ with $\tilde{k}_0\ge K$, then we have
	\begin{align}\label{eq-a38}
		\bq_3^{\tilde{k}_0 + 1} \leq \bJ^{[3]}
		\bq_3^{\tilde{k}_0}.
	\end{align}
Further, if
\begin{align}\label{ineq-eps-delta}
	\kappa_F \epsilon^{\tilde{k}_0} + \delta^{\tilde{k}_0} \leq \frac{1}{{20}}\sigma^{\frac{m}{2}}
\end{align}
and the parameters $\alpha$, $m$, $c_k$, and $\gamma$ satisfy \eqref{ineq-cond}, then \eqref{eq-a38} implies that
\begin{align}\label{eq-a39}
	u_3^{\tilde{k}_0+1} \leq \sigma^{\frac{m}{2}}u_3^{\tilde{k}_0}.
\end{align}
\end{props}
\begin{proof}
	See Appendix \ref{prf-theom-3}.
\end{proof}
\subsubsection{Proof of Theorem \ref{theo-3}}
According to Proposition \ref{prop-4}, to prove Theorem \ref{theo-3}, we only need to show \eqref{ineq-M1M2} and \eqref{ineq-eps-delta} hold for all $k\geq K$. This is done in  Proposition \ref{prop-5}.
\begin{props}\label{prop-5}
	Under the setting of Theorem \ref{theo-3},
	\eqref{ineq-M1M2} and \eqref{ineq-eps-delta} hold for any $k \geq K$.
\end{props}
\begin{proof}
	See Appendix \ref{prf-prop-5}.
\end{proof}

By combining Propositions \ref{prop-4} and \ref{prop-5}, we complete the proof of Theorem \ref{theo-3}.

\begin{figure*}[tb]
	\begin{align}\label{eq-J3}
		\hspace{-3mm}\bq_3^k \!\triangleq \! \begin{pmatrix}
			\|\bx^{k} \!-\! \bW^\infty \bx^{k} \| \\
			\frac{1}{L_1} \|\bg^{k} \!-\! \bW^\infty \bg^{k} \| \\
			\sqrt{n} \|\ox^{k} \!-\! x^* \|\end{pmatrix}\!,~
		\bJ^{[3]}\!\triangleq\! \begin{bmatrix}
			\sigma^m \left(1 + \alpha \kappa_F \epsilon^{\tilde{k}_0} \right) & \sigma^m \alpha (1 + \epsilon^{\tilde{k}_0}) \kappa_F & 2 \sigma^m \alpha \epsilon^{\tilde{k}_0} \kappa_F \\
			\sigma^m \left( 2 \!+\! \alpha \kappa_F \!+\! 2\alpha \kappa_F \epsilon^{\tilde{k}_0} \right) & \sigma^m \left(1 \!+\!\alpha \kappa_F(\sigma^m \!+\! 2\epsilon^{\tilde{k}_0}) \right) & \sigma^m \alpha \left(1 \!+\! 2\kappa_F \epsilon^{\tilde{k}_0} \!+\! \delta^{\tilde{k}_0} \right) \\
			\alpha \kappa_F \left(1 + \epsilon^{{\tilde{k}_0}} \right) & \alpha \kappa_F \epsilon^{{\tilde{k}_0}} & 1 - \alpha + \alpha \delta^{\tilde{k}_0} + \alpha \kappa_F \epsilon^{{\tilde{k}_0}}
		\end{bmatrix},
	\end{align}
	where  $\kappa_F =\frac{ L_1 }{\mu}$,
	$\delta^k \triangleq \frac{ L_2}{2\mu} \|\ox^k - x^*\|$, and $\epsilon^{k} \triangleq \frac{1}{M_1} \left( \frac{L_2}{\sqrt{n}} \|\bx^{k} - \bW^\infty \bx^{k} \| + \frac{1}{\sqrt{n}}\|\bH^{k} - \bW^\infty \bH^{k}\|_F+{c_k\mu}\right)$ for all $k\ge 0$.
\end{figure*}

\section{Numerical Experiments}
In the numerical experiments, we consider quadratic programming and logistic regression problems over a network. The network is randomly generated with $n$ nodes connected by $\frac{\tau n(n-1)}{2}$  edges, where $\tau \in(0,1] $ is the connectivity ratio.
We pre-compute the optimal solution $x^*$ with a centralized Newton's method. The performance metric is the relative error, defined as $\frac{1}{n}{\|\mathbf{x}^{k}-\mathbf{x}^{*}\|^2}/{\|\mathbf{x}^{0}-\mathbf{x}^{*}\|^2}$.
We compare the proposed method with the first-order method ABC \cite{xu2021distributed}, the multi-step consensus accelerated first-order method Mudag \cite{ye2020multi}, and the second-order method SONATA \cite{scutari2019distributed}.
We use hand-optimized step sizes for all the algorithms.  The experiments are done on a laptop with an Intel(R) Core(TM) i7 CPU running at 1.80GHz, 16.0 GB of RAM, and Windows 10 operating system.

\subsection{Quadratic Programming}
We demonstrate that the convergence rate of the proposed Newton's method is independent of the condition number $\kappa_F$. Consider solving a quadratic programming problem over a network, i.e.,
\begin{equation*}
x^*=\arg\min_{x \in \mathbb{R}^d} \sum_{i=1}^{n}\left(\frac{1}{2} x^T Q_{i} x+p_{i}^T x \right).
\end{equation*}
Each node $i$ has private data $Q_{i} \in \mathbb{R}^{d \times d}\succ 0$ and $p_{i} \in \mathbb{R}^{d}$, whose elements are generated according to the standard Gaussian distribution.  To show the independence of the condition number $\kappa_{F}$, we generate matrices $Q_1$, $\ldots$, $Q_n$ under different condition numbers $\kappa_{F}=10$, $10^2$, $10^4$. We set $n=10$, $\tau=0.2$, and $d=30$. We compare the proposed Newton's method with ABC. For the proposed Newton's method, to further show that its convergence rate is connected with the number of multi-step consensus inner-loops, we set $m=15$ and $m=20$. To show the asymptotic rate with $m\to \infty$, we consider $\kappa_{F}=10^4$ and $m=k$, where $k$ is the index of iteration. We use Rank-$3$ compression, set $\gamma=0.03$, and use the increasing step sizes $\alpha_k=\min\{1, 0.02 \times 1.1^k \}$.
In ABC, we set the step sizes $\alpha=0.4$, $0.05$, $0.0008$ for the condition numbers $\kappa_F=10$, $10^2$, $10^4$, respectively.

Fig. \ref{fig-kappa}  shows the relative error versus the number of iterations. For the first-order method ABC, the convergence becomes slower with a larger $\kappa_{F}$. However, for each fixed $m$, our proposed Newton's method converges at the same rate under different $\kappa_{F}$. Also, the convergence rate is faster with a larger $m$.  When we increase the number of multi-step consensus inner-loops to $m=k$, the convergence rate becomes much faster. These experiment results corroborate the theoretical findings.


\begin{figure}[t]
	\centering
	\includegraphics[width=0.45\textwidth]{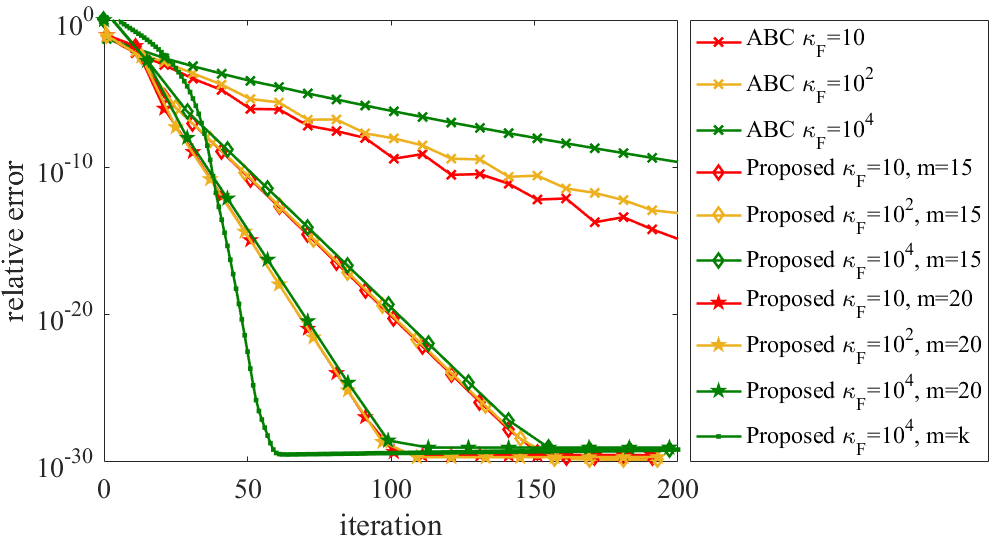}
	\caption{Robustness to $\kappa_{F}$}
	\label{fig-kappa}
\end{figure}

\subsection{Logistic Regression}
To further show the efficiency of the proposed Newton's method, we solve a logistic regression problem in the form of
\begin{align}\nonumber
{{x}}^{*} = \underset{{x} \in \mathbb{R}^{d}}{\operatorname{argmin}} ~ \frac{\rho}{2}\|{x}\|^{2}+\sum_{i=1}^{n} \sum_{j=1}^{m_{i}} \ln \left(1+\exp \left(-\left(\mathbf{o}_{i j}^{T} {x}\right) \mathbf{p}_{i j}\right)\right),
\end{align}
where each node $i$ privately owns $m_i$ training samples $(\mathbf{o}_{i j}, \mathbf{p}_{i j}) \in \mathbb{R}^{d} \times\{-1,+1\}$, $j=1, \ldots, m_{i}$. The elements of $\mathbf{o}_{i j}$ are randomly generated following the standard Gaussian distribution and those of $\mathbf{p}_{i j}$ are generated following the uniform distribution on $\{-1,1\}$. The regularization term $\frac{\rho}{2}\|{x}\|^{2}$ parameterized by $\rho>0$ is to avoid overfitting.

The settings are as follows. There are $n = 30$ nodes with connectivity ratio $\tau=0.2$. Each node has $100$ samples, i.e. $m_i = 100, \forall i$. We set the dimension $d =20$ and the regularization parameter $\rho=0.001$.

\begin{figure*}[htbp] 	
	\centering	
	\subfigure[]{
		\begin{minipage}[t]{0.32\linewidth}
			\centering
			\includegraphics[height=1.5in]{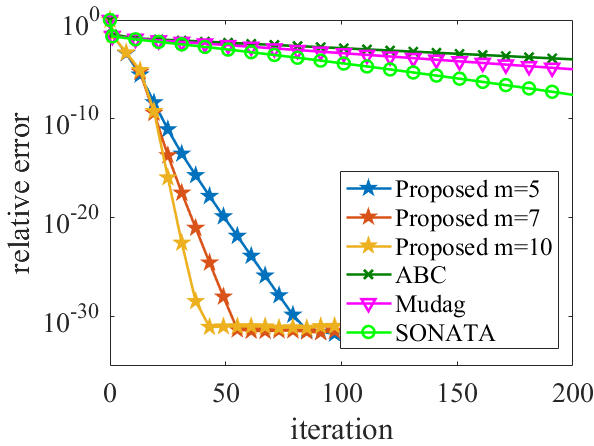}
		\end{minipage}%
	}%
	\subfigure[]{
		\begin{minipage}[t]{0.32\linewidth}
			\centering
			\includegraphics[height=1.5in]{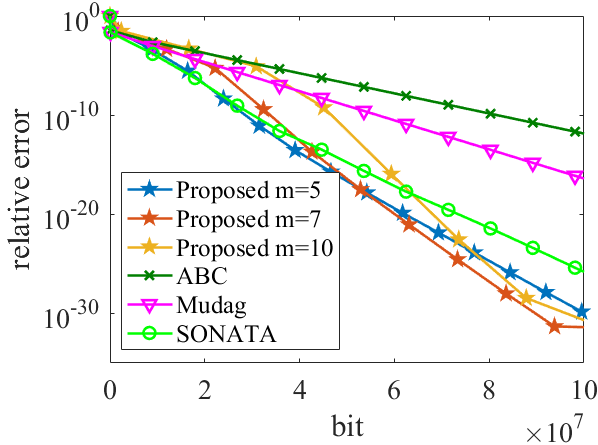}
		\end{minipage}%
	}%
	\centering
	\subfigure[]{
		\begin{minipage}[t]{0.32\linewidth}
			\centering
			\includegraphics[height=1.5in]{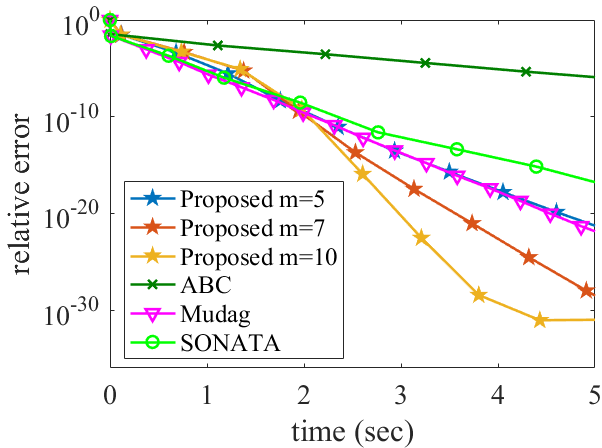}
		\end{minipage}%
	}%
	\caption{Top-$K$ for logistic regression}
	\label{fig-top}
\end{figure*}
\begin{figure*}[htbp] 	
	\centering	
	\subfigure[]{
		\begin{minipage}[t]{0.32\linewidth}
			\centering
			\includegraphics[height=1.5in]{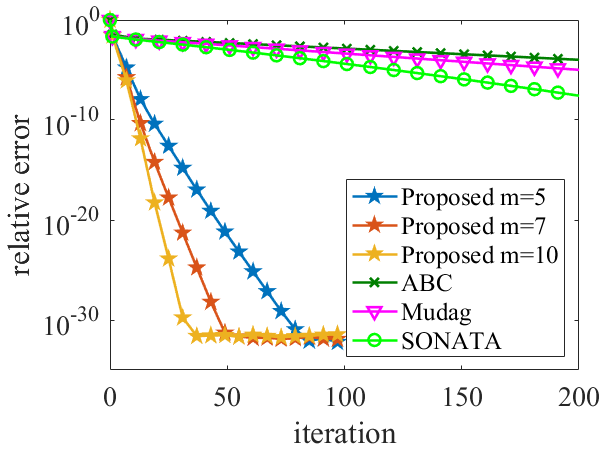}
		\end{minipage}%
	}%
	\subfigure[]{
		\begin{minipage}[t]{0.32\linewidth}
			\centering
			\includegraphics[height=1.5in]{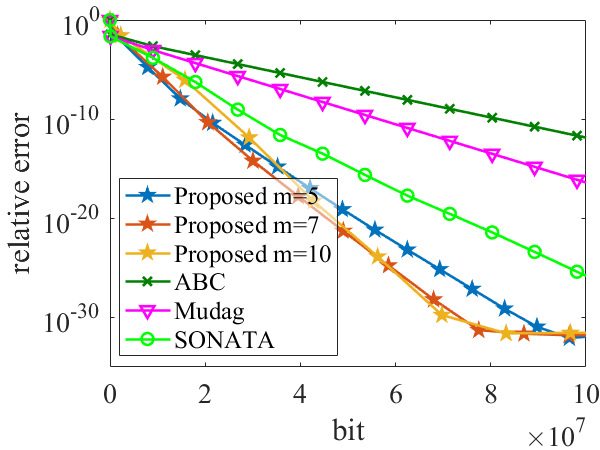}
		\end{minipage}%
	}%
	\centering
	\subfigure[]{
		\begin{minipage}[t]{0.32\linewidth}
			\centering
			\includegraphics[height=1.5in]{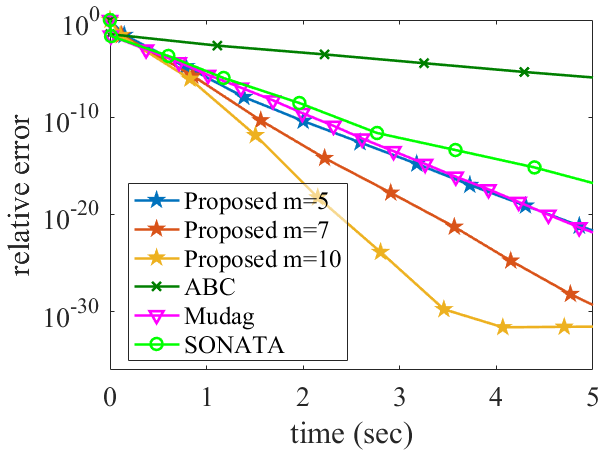}
		\end{minipage}%
	}%
	\caption{Rank-$K$ for logistic regression}
	\label{fig-lowrank}
\end{figure*}

We compare our proposed method with ABC, Mudag, and SONATA. The step size of ABC is set to $\alpha=0.9$. The step sizes of Mudag are set to the theoretically optimal values suggested in \cite{ye2020multi}. For SONATA, we use the second-order approximation of the local objective such that $f_i(x)\approx f_i(x_i^k)+\langle g_i^k, x-x_i^k \rangle +\frac{1}{2}(x-x_i^k)^T (\nabla^2 f_i(x_i^k)+\epsilon I_d)(x-x_i^k)$, where $g_i^k$ is the gradient approximation of node $i$ at the $k$-th iteration. Here, we set $\epsilon=0.45$.

We test the proposed method using both the Rank-$K$ and Top-$K$ compression operators.
For the Top-$K$ compression operator, node $i$ transmits $K=20$ entries of the matrix with largest absolute values. We use the increasing step sizes $\alpha_k=\min\left\{1,0.2 \times 1.1^k\right\}$. We set $\gamma=0.06$.
For the Rank-$K$ compression operator, node $i$ performs a singular value decomposition on the matrix and transmits the largest $K=3$ singular values as well as the corresponding singular vectors. We use the increasing step sizes $\alpha_k=\min\left\{1,0.1 \times 1.1^k\right\}$. We set $\gamma=0.06$. For both experiments, we set $M=0$ and the algorithm still works well. We speculate that this is because we use the CG step to avoid computing the inverse of $H_i^k+MI_d$ and $H_i$ becomes closer and closer to  $\oH^k$ as the iteration proceeds.

Fig. \ref{fig-top} and Fig. \ref{fig-lowrank}  illustrate the relative error versus the number of iterations, the number of bits for communication and the running time of the Top-$K$  and Rank-$K$ compression operators, respectively. We run the proposed method with different $m$, i.e., different numbers of multi-step consensus inner-loops. From Fig. \ref{fig-top} and Fig. \ref{fig-lowrank}, we can observe that the proposed Newton's method has the best performance. Mudag outperforms ABC because Mudag uses multi-step consensus and acceleration. SONATA outperforms Mudag due to the use of the local Hessian $\nabla^2 f_i(x_i^k)$ to approximate the global Hessian. The proposed Newton's method outperforms SONATA since we use DAC to track the global Hessian. Note that our communication cost is acceptable, thanks to the introduction of a compression procedure.

The convergence rate of the proposed Newton's method scales with $m$,  which validates the theoretical result in Theorem \ref{theo-3}. We can also see $m$ balances between the number of iterations and the transmitted bits. A larger $m$ improves the iteration complexity but increases the communication cost.

\section{Conclusions}
This paper considers a finite-sum minimization problem over a decentralized network.
We propose a communication-efficient decentralized Newton's method for solving it, which has provably faster convergence than first-order algorithms. Multi-step consensus that balances between computation and communication is used for communicating local copies of the decision variable and gradient approximations. We also use compression with error compensation for transmitting the local Hessian approximations, which utilizes the global second-order information while avoiding high communication cost.
We present a novel convergence analysis and obtain a theoretically faster convergence rate than those of first-order algorithms. 
{One future direction is to develop stochastic second-order algorithms with provably $\kappa_F$-independent super-linear convergence rate, considering the case when computing the local full gradient and Hessian is not affordable on each node. Another interesting direction is to develop decentralized second-order algorithms to solve nonconvex and nonsmooth problems that arise in various machine learning and signal processing applications.  }

\section{Appendix}\label{prf-global-linear}

\subsection{Proof of Proposition \ref{prop-compression operator}}\label{app-compression operator}
\begin{proof}	
For the Rank-$K$ compression operator, we compute
\begin{align}\label{eq-Q1}
\begin{aligned}
\|\mathcal{Q}(A)-A\|_F^2
=&\big\|\sum\nolimits_{i={K+1}}^{d} \sigma_i u_iv_i^T\big\|_F^2\\
\le& (d-K) \sigma_{K+1}^2\le\frac{d-K}{K}\|\mathcal{Q}(A)\|_F^2.
\end{aligned}
\end{align}
Then, we have
\begin{align}\label{eq-Q2}
\hspace{-4mm}\|A\|_F^2\!=\!\|\mathcal{Q}(A)\!-\!A\|_F^2\!+\!\|\mathcal{Q}(A)\|_F^2\!\ge\! \frac{d}{d\!-\!K}\|\mathcal{Q}(A)\!-\!A\|_F^2,
\end{align}
where we substitute \eqref{eq-Q1} in the  inequality.
Thus, \eqref{eq-Q2} implies  that
\begin{align} \label{eq-new-01}
\|\mathcal{Q}(A)-A\|_F\le \sqrt{1-\frac{K}{d}}\|A\|_F\le \left(1-\frac{K}{2d}\right)\|A\|_F.
\end{align}
For the Top-$K$ compression operator, let $\{a_1, a_2, \ldots, a_{d^2}\}$ be the entries of $A$ sorted by their absolute values in a descending order. Then, we have
\begin{align*}
		\|\mathcal{Q}(A)-A\|_F^2\!
		=\!\sum_{i = K\!+\!1}^{d^2}\! a_i^2
		\!\le\! (d^2-K) a_{K+1}^2 \le\frac{d^2-K}{K}\|\mathcal{Q}(A)\|_F^2.
\end{align*}
Following a similar argument for deriving \eqref{eq-new-01}, for the Top-$K$ compression operator, we have
\begin{align}
\hspace{-2mm}	\|\mathcal{Q}(A)-A\|_F\le \sqrt{1-\frac{K}{d^2}}\|A\|_F\le \left(1-\frac{K}{2d^2}\right)\|A\|_F.
\end{align}
This completes the proof.
\end{proof}

\subsection{Preliminary}
This section gives some preliminaries that are useful in the ensuing convergence analysis.
The following lemma bounds the consensus errors of the iterate $\bx^k$ and the gradient approximation $\bg^k$.
\begin{lem}\label{lem-preliminary-x-g}
Under Assumptions \ref{asm-W} and \ref{asp:lipschitz-continuous}, for all $k\ge 0$, we have
\begin{align}\label{ineq:xkplus1}
\hspace{-5mm}\|\bx^{k + 1}\!\! -\!\! \bW^\infty \bx^{k + 1} \|\! \leq\! \sigma^m ( \|\bx^k\! \!-\!\! \bW^\infty \bx^k\| \!+\! \alpha \|\bd^k\! \!-\!\! \bW^\infty \bd^k\|)
\end{align}
and
\begin{align}\label{ineq:gkplus1}
\hspace{-3mm}\|\bg^{k + 1} \!-\! \bW^\infty \bg^{k + 1} \|\! \leq\! \sigma^m ( \|\bg^k \!-\! \bW^\infty \bg^k\| \!+\! L_1 \|\bx^{k + 1} \!-\! \bx^k\|).
\end{align}
\end{lem}

\begin{proof}
First, we show that multi-step consensus with $m$ inner loops improves the convergence rate of the consensus error from $\sigma$ to $\sigma^m$. To do this, we compute
\begin{align*}
\bW^m = (\bW - \bW^\infty + \bW^\infty)^m = (\bW - \bW^\infty)^m +\bW^\infty,
\end{align*}
where we use $(\bW^\infty)^2 = \bW^\infty$ and $\bW \bW^\infty = \bW^\infty \bW = \bW^\infty$. Thus, we have
\begin{align*}
\|\bW^m - \bW^\infty\| = \|\bW - \bW^\infty\|^m \leq \sigma^m.
\end{align*}
Then, $\bx^{k + 1} = \bW^m  (\bx^k - \alpha \bd^k)$ implies that
\begin{align*}
\begin{aligned}
\bx^{k + 1}\!-\!\bW^\infty \bx^{k + 1}\! =& (\bW^m - \bW^\infty)(\bx^k - \alpha \bd^k) \\
=& (\bW^m\! - \!\bW^\infty)\big(\bx^k \!- \!\alpha \bd^k \!-\! \bW^\infty (\bx^k\! -\! \alpha \bd^k)\big).
\end{aligned}
\end{align*}
Taking the norm $\|\cdot\|$ on both sides of the above equality and using the triangle inequality, we get \eqref{ineq:xkplus1}.

Second, according to  $\bg^{k + 1} = \bW^m (\bg^k + \nabla f(\bx^{k + 1}) - \nabla f(\bx^k))$, we have
\begin{align*}
\begin{split}
&\bg^{k + 1} - \bW^\infty \bg^{k + 1} \\
= & (\bW^m - \bW^\infty) \big(\bg^k - \bW^\infty \bg^k+\nabla f(\bx^{k + 1}) - \nabla f(\bx^{k })\big).
\end{split}		
\end{align*}
Taking the norm $\|\cdot\|$ on both sides of the above inequality and using the triangle inequality and Assumption \ref{asp:lipschitz-continuous}, we get \eqref{ineq:gkplus1} and complete the proof.
\end{proof}

Since node $i$ uses $g_i^k$ to estimate the global gradient, we expect $g_i^k$ to be zero when the algorithm achieves optimality. For a convergent algorithm, the iterates should be fixed such that $\bx^{k+1}$ is equal to $\bx^k$.  These observations motivate us to bound the norm of the gradient approximation $\bg^k$ and the difference between two successive iterates.
\begin{lem}\label{lem:xkplus1-xk}
Under Assumptions \ref{asm-W} and \ref{asp:lipschitz-continuous}, if $c_k\le 1$ and condition \eqref{eq-M1M2} holds for a certain $k_0$, then we have
\begin{align}\label{ineq:gknorm}
\begin{split}
\hspace{-4mm}\|\bg^{k}\|\!\leq\! \|\bg^{k}\! \!-\!\! \bW^\infty \bg^{k}\| \!+\! L_1 \|\bx^{k} \!-\! \bW^\infty \bx^{k}\| \! \!+\!\! \sqrt{n} \| \nabla F(\ox^{k})\|
\end{split}
\end{align}
 for all $k\ge 0$
and
\begin{align}\label{eq-11}
\hspace{-2mm}\|\bx^{k_0 + 1} \!-\! \bx^{k_0} \|
&\leq \!\left( 2 + \frac{2 \alpha L_1}{M_1} \right) \|\bx^{k_0} \!-\! \bW^\infty \bx^{k_0}\|  \\
& +\! \frac{2 \alpha}{M_1} \|\bg^{{k_0}} \!-\! \bW^\infty \bg^{{k_0}}\|\! +\! \frac{2 \alpha \sqrt{n}}{M_1} \| \nabla F(\ox^{{k_0}})\|.\nonumber
\end{align}
\end{lem}

\begin{proof}
From the update $\bg^{k + 1} = \bW^m (\bg^k + \nabla f(\bx^{k + 1}) - \nabla f(\bx^{k}) )$, we have  $\og^{k} = \overline{ \nabla f} (\bx^{k})$.  Under Assumption \ref{asp:lipschitz-continuous}, we have
\begin{align}\label{eq-barg-grad}
\|\og^{k} - \nabla F(\ox^{k})\| \leq \frac{L_1}{\sqrt{n}} \|\bx^{k} - \bW^\infty \bx^{k} \|,
\end{align}
which implies that
\begin{align}\label{eq-14}
\hspace{-3mm}	\|\bg^{k}\| \!\leq & \|\bg^{k} - \bW^\infty \bg^{k}\| + \sqrt{n} \|\og^{k}\| \\
\leq & \|\bg^{k} \!-\! \bW^\infty \bg^{k}\| \!+ \!\sqrt{n} \|\og^{k} \!-\! \nabla F(\ox^{k})\| \!+ \!\sqrt{n} \| \nabla F(\ox^{k})\| \nonumber\\
\leq & \|\bg^{k} \!-\! \bW^\infty \bg^{k}\| \!+\! L_1 \|\bx^{k} - \bW^\infty \bx^{k}\|  + \sqrt{n} \| \nabla F(\ox^{k})\|.\nonumber
\end{align}
This inequality gives \eqref{ineq:gknorm}.

If \eqref{eq-M1M2} holds for a certain $k_0$, according to the fact that  $( \diag\{H_i^{k + 1}\}+M I_{nd}) \bd^{k + 1} = \bg^{k + 1} + \br^{k + 1}$, we have
\begin{align}\label{ineq:dknorm-new}
	\|\bd^{k_0}-\bW^{\infty}\bd^{k_0}\|\le\|\bd^{k_0}\| \leq \frac{\| \bg^{k_0}+\br^{k_0}\|}{M_1} \leq \frac{2\|\bg^{k_0}\|}{M_1},
\end{align}
where we use $c_k\le1$ in the last inequality and the first inequality is given for later use (see \eqref{eq-diff}).
According to $\bx^{k + 1} = \bW^m  (\bx^k - \alpha \bd^k)$, we have
\begin{align}\label{eq-12}
\hspace{-3mm}\|\bx^{{k_0} + 1} \!-\! \bx^{k_0} \|
\leq&  \|(\bW^m \!-\! I_{nd}) (\bx^{k_0} - \bW^\infty \bx^{k_0})\| + \alpha \|\bd^{k_0}\| \nonumber\\
\leq&  2 \|\bx^{k_0} - \bW^\infty \bx^{k_0}\| + \frac{2 \alpha}{M_1} \|\bg^{k_0}\|,	
\end{align}
where we substitute \eqref{ineq:dknorm-new} in the last inequality.
By substituting \eqref{eq-14} into \eqref{eq-12}, we get \eqref{eq-11} and complete the proof.
\end{proof}

\subsection{Proof of Proposition \ref{prop-1}}\label{prf-theorem1}
\begin{proof}
First, we prove \eqref{eq-rate-1} in three steps. We will bound the consensus error $\|\bx^{{k_0} + 1} - \bW^\infty \bx^{{k_0} + 1} \|^2 $,  the gradient tracking error $\frac{1}{L_1^2} \|\bg^{{k_0} + 1} - \bW^\infty \bg^{{k_0} + 1} \|^2 $, and the network optimality gap $\frac{n}{L_1} (F(\ox^{{k_0} + 1}) - F(x^*))$ in Steps I, II, and III, respectively.

{\bf Step I:} The following lemma bounds the consensus error $\|\bx^{{k_0} + 1} - \bW^\infty \bx^{{k_0} + 1} \|^2 $.
\begin{lem}\label{lem-3}
Under Assumptions  \ref{asm-W} and \ref{asp:lipschitz-continuous}, if $c_k\le 1$, condition  \eqref{eq-M1M2} holds for a certain $k_0$, and  $\alpha$ satisfy \eqref{eq-a-c}, then we have
\begin{align}\label{eq-cons-x-J}
\begin{split}
&\|\bx^{{k_0}+1}-\bW^{\infty}\bx^{{k_0}+1}\|^2\\
\le& \bJ^{[1]}_{11}\|\bx^{{k_0}}-\bW^{\infty}\bx^{{k_0}}\|^2+ \bJ^{[1]}_{12}\cdot \frac{1}{L_1^2}\|\bg^{k_0}-\bW^{\infty}\bg^{k_0}\|^2\\
&+ \bJ^{[1]}_{13}\cdot\frac{n}{L_1}\left(F(\ox^{k_0})-F(x^*)\right).
\end{split}
\end{align}
\end{lem}

\begin{proof}
With \eqref{ineq:xkplus1}, we have
\begin{align}\label{eq-diff}
&\|\bx^{{k_0} + 1} - \bW^\infty \bx^{{k_0} + 1} \|^2\nonumber\\
\leq & \sigma^{2m} \big( (1 + \eta_1)\|\bx^{k_0} - \bW^\infty \bx^{k_0}\|^2 \nonumber\\
&+ \left(1 + \frac{1}{\eta_1}\right) \alpha^2 \|\bd^{k_0} - \bW^\infty \bd^{k_0}\|^2 \big) \\
\leq & \frac{\sigma^{2m - 2} (1 + \sigma^2)}{2} \|\bx^{k_0} \!-\! \bW^\infty \bx^{k_0}\|^2\! +\! \frac{8 \sigma^{2m} \alpha^2 }{(1 \!-\! \sigma^2) M_1^2} \|\bg^{k_0}\|^2,\nonumber
\end{align}
where the first inequality holds for any $\eta_1>0$ due to Young's inequality and the second inequality holds by setting $\eta_1= \frac{1 - \sigma^2}{2\sigma^2}$  and substituting \eqref{ineq:dknorm-new}.

Further, according to \eqref{ineq:gknorm}, we have
\begin{align}\label{eq-g}
&\|\bg^k\|^2 \nonumber\\
\leq & 4 \|\bg^k \!-\! \bW^\infty \bg^k\|^2 \!+\! 2 L_1^2 \|\bx^{k} \!-\! \bW^\infty \bx^{k} \|^2 \!+\! 4 n \|\nabla F(\ox^k)\|^2 \nonumber\\
\leq & 4 \|\bg^k - \bW^\infty \bg^k\|^2 + 2 L_1^2 \|\bx^{k} - \bW^\infty \bx^{k} \|^2 \\
&+ 8 L_1 n\big(F(\ox^{k}) - F(x^*)\big)\nonumber
\end{align}
for all $k\ge 0$, where the last inequality holds since
\begin{align}\label{ineq:gradfxk}
\| \nabla F(\ox^{k})\|^2 \leq 2 L_1 \left(F(\ox^{k}) - F(x^*)\right),
\end{align}
whose proof can be found in  \cite[Theorem 2.1.5]{nesterov2003introductory}.
By substituting \eqref{eq-g} into \eqref{eq-diff}, we get
\begin{align}\label{eq-cons-x}
&\|\bx^{{k_0} + 1} - \bW^\infty \bx^{{k_0} + 1} \|^2 \nonumber\\
\leq & { \sigma^{2m - 2}} \left(\frac{1 + \sigma^2}{2} + \frac{16 \sigma^{2} \alpha^2 L_1^2}{(1 - \sigma^2) M_1^2} \right) \|\bx^{k_0} - \bW^\infty \bx^{k_0}\|^2 \nonumber\\
& + \frac{32 \sigma^{2m} \alpha^2 L_1^2}{(1 - \sigma^2) M_1^2} \cdot \frac{1}{L_1^2} \|\bg^{k_0} - \bW^\infty \bg^{k_0}\|^2 \\
&+ \frac{64 \sigma^{2m} \alpha^2 L_1^2}{(1 - \sigma^2) M_1^2} \cdot\frac{n}{L_1}\left(F(\ox^{{k_0}}) - F(x^*)\right).\nonumber
\end{align}
By choosing $\alpha \leq \frac{M_1^2 (1 - \sigma^2)^3}{100 L_1 M_2{\sigma^{m-1}}}$, we have
$$
\begin{aligned}
&\sigma^{2m - 2} \left(\frac{1 + \sigma^2}{2} + \frac{16 \sigma^{2} \alpha^2 L_1^2}{(1 - \sigma^2) M_1^2} \right)\\
\leq& \frac{1 + \sigma^2}{2} + {\sigma^{2m-2}}\cdot\frac{16 \sigma^{2} M_1^4 (1 - \sigma^2)^6 L_1^2}{10^4 L_1^2 M_2^2 {\sigma^{2m-2}} (1 - \sigma^2) M_1^2} \\
\leq& \frac{1 + \sigma^2}{2} + 0.01(1 - \sigma^2) = 1 - 0.49(1 - \sigma^2)= \bJ^{[1]}_{11}
\end{aligned}
$$
and
$$
\begin{aligned}
\frac{32 \sigma^{2m} \alpha^2 L_1^2}{(1 - \sigma^2) M_1^2} &\leq \frac{32 \sigma^{2m} M_1^4 (1 - \sigma^2)^6L_1^2}{10^4 L_1^2 M_2^2 {\sigma^{2m-2}} (1 - \sigma^2) M_1^2}\\
& \leq { 0.004(1 - \sigma^2)^3}= \bJ^{[1]}_{12}
\end{aligned}
$$
and $$\frac{{ 64} {\sigma^{2m}} \alpha^2 L_1^2}{(1 - \sigma^2) M_1^2}= \bJ^{[1]}_{13}.$$
Thus, we get \eqref{eq-cons-x-J} and complete the proof.
\end{proof}

{\bf Step II:} The following lemma bounds the gradient tracking error $	\frac{1}{L_1^2} \|\bg^{{k_0} + 1} - \bW^\infty \bg^{{k_0} + 1} \|^2 $.
\begin{lem}\label{lem-4}
Under the setting of Lemma \ref{lem-3}, we have
\begin{align}\label{eq-g-consensus-J}
\begin{split}
&\frac{1}{L_1^2}	\|\bg^{{k_0} + 1} - \bW^\infty \bg^{{k_0} + 1} \|^2 \\
\leq &  \bJ^{[1]}_{21} \|\bx^{k_0} - \bW^\infty \bx^{k_0}\|^2+\bJ^{[1]}_{22}\cdot \frac{1}{L_1^2} \|\bg^{k_0} - \bW^\infty \bg^{k_0}\|^2 \\
& + \bJ^{[1]}_{23}\cdot \frac{n}{L_1}\left(F(\ox^{{k_0}}) - F(x^*)\right).
\end{split}
\end{align}
\end{lem}

\begin{proof}
With \eqref{ineq:gkplus1}, we have
\begin{align}\label{eq-22}
&\|\bg^{k + 1} - \bW^\infty \bg^{k + 1} \|^2\\ 
\leq & \sigma^{2m}\!\left((1 \!+\! \eta_2)\|\bg^k \!-\! \bW^\infty \bg^k\|^2 \!+\! \left(1 + \frac{1}{\eta_2}\right) L_1^2 \|\bx^{k + 1} \!-\! \bx^k\|^2 \right)\nonumber\\
\leq & \frac{\sigma^{2m - 2} (1 + \sigma^2)}{2} \|\bg^k - \bW^\infty \bg^k\|^2 +  \frac{2 \sigma^{2m} L_1^2}{1 - \sigma^2} \|\bx^{k + 1} - \bx^k\|^2 \nonumber
\end{align}
for all $k\ge 0$, where the first inequality holds for any $\eta_2>0$ due to Young's inequality and the second inequality holds by setting $\eta_2= \frac{1 - \sigma^2}{2\sigma^2}$.

Further, according to \eqref{eq-11}, we have
\begin{align}\label{eq-23}
& \|\bx^{{k_0} + 1} - \bx^{k_0} \|^2 \nonumber\\ \leq & 2 \left( 2 + \frac{2 \alpha L_1}{M_1} \right)^2 \|\bx^{k_0} - \bW^\infty \bx^{k_0}\|^2 \nonumber\\
&+ \frac{16 \alpha^2}{M_1^2} \|\bg^{{k_0}} - \bW^\infty \bg^{{k_0}}\|^2 + \frac{16 \alpha^2 n}{M_1^2} \| \nabla F(\ox^{{k_0}})\|^2 \nonumber\\ 
\leq & 2 \left( 2 + \frac{2 \alpha L_1}{M_1} \right)^2 \|\bx^{k_0} - \bW^\infty \bx^{k_0}\|^2\\
& + \frac{16 \alpha^2}{M_1^2} \|\bg^{{k_0}} - \bW^\infty \bg^{{k_0}}\|^2 + \frac{32 \alpha^2 L_1}{M_1^2} n (F(\ox^{{k_0}}) - F(x^*)),\nonumber
\end{align}
where we substitute \eqref{ineq:gradfxk} in the last inequality.
By substituting \eqref{eq-23} into \eqref{eq-22}, we have
\begin{align}\label{eq-g-consensus}
		&\frac{1}{L_1^2}\|\bg^{{k_0} + 1} - \bW^\infty \bg^{{k_0} + 1} \|^2 \nonumber\\
		\leq & {\sigma^{2m - 2} } \left( \frac{1 + \sigma^2}{2} +  \frac{32 \alpha^2 \sigma^{2} L_1^2}{(1 - \sigma^2) M_1^2} \right) \cdot \frac{1}{L_1^2} \|\bg^{k_0} - \bW^\infty \bg^{k_0}\|^2 \nonumber\\
		& + \frac{2 \sigma^{2m}}{1 - \sigma^2} \cdot 2 \left( 2 + \frac{2 \alpha L_1}{M_1} \right)^2  \|\bx^{k_0} - \bW^\infty \bx^{k_0}\|^2  \\
		&+ \frac{2 \sigma^{2m} }{1 - \sigma^2}\cdot\frac{32 \alpha^2 L_1^2}{M_1^2}\cdot\frac{n}{L_1}\left(F(\ox^{{k_0}}) - F(x^*)\right).\nonumber
\end{align}
By substituting $\alpha \leq \frac{M_1^2 (1 - \sigma^2)^3}{100 L_1 M_2{\sigma^{m-1}} }$ into \eqref{eq-g-consensus}, we have
$$
\begin{aligned}
&\sigma^{2m - 2} \left(\frac{1 + \sigma^2}{2} + \frac{32 \sigma^{2} \alpha^2 L_1^2}{(1 - \sigma^2) M_1^2} \right) \\
\leq&\frac{1 + \sigma^2}{2} + \frac{{\sigma^{2m-2}}32 \sigma^{2} M_1^4 (1 - \sigma^2)^6 L_1^2}{10^4 L_1^2 M_2^2{\sigma^{2m-2}} (1 - \sigma^2) M_1^2}\\
\leq&\frac{1 + \sigma^2}{2} + 0.01(1 - \sigma^2)= 1 - 0.49(1 - \sigma^2)= \bJ^{[1]}_{22}
\end{aligned}
$$
and
$$
\begin{aligned}
&\frac{2 \sigma^{2m}}{1 - \sigma^2} \cdot 2 \left( 2 + \frac{2 \alpha L_1}{M_1} \right)^2\leq \frac{2 \sigma^{2m}}{1 - \sigma^2} \cdot 2 \left( 8 + \frac{8 \alpha^2 L_1^2}{M_1^2} \right)\\
&\leq \frac{32\sigma^{2m} }{1 - \sigma^2}+\frac{32\sigma^{2m}L_1^2}{(1-\sigma^2)M_1^2}\cdot\frac{M_1^4(1-\sigma^2)^6}{10^4L_1^2M_2^2\sigma^{2m-2}}
\leq \frac{{33}}{1 - \sigma^2}= \bJ^{[1]}_{21}
\end{aligned}
$$
and
$$ \frac{2 \sigma^{2m} }{1 - \sigma^2}\cdot\frac{32 \alpha^2 L_1^2}{M_1^2} = \frac{{ 64}{\sigma^{2m}}\alpha^2 L_1^2 }{(1 - \sigma^2) M_1^2}= \bJ^{[1]}_{23}.$$
Thus, we get \eqref{eq-g-consensus-J} and complete the proof.
\end{proof}
{\bf Step III:} The following lemma bounds  the network optimality gap $\frac{n}{L_1} (F(\ox^{k_0 + 1}) - F(x^*))$.
\begin{lem}\label{lem-5}
Under the setting of Lemma \ref{lem-3}, we have
\begin{align}\label{eq-Fx*-J}
\begin{split}
&\frac{n}{L_1}\left(F(\ox^{{k_0} + 1}) - F(x^*)\right)\\
\leq&
\bJ^{[1]}_{31}\|\bx^{k_0} -\bW^\infty \bx_{k_0}\|^2 + \bJ^{[1]}_{32}\cdot \frac{1}{L_1^2}\|\bg^{k_0} - \bW^\infty \bg^{k_0}\|^2 \\
&+\bJ^{[1]}_{33} 	\cdot\frac{n}{L_1}\left(F(\ox^{{k_0}}) - F(x^*)\right).
\end{split}
\end{align}
\end{lem}
\begin{proof}
Let us denote $B_i^k = (H^k_i + M I_d)^{-1}$ and $\oB^k = \frac{1}{n} \sum_{i = 1}^n B^k_i$. Since $\ox^{k + 1} = \ox^{k} - \alpha \od^k $, under Assumption \ref{asp:lipschitz-continuous} and \eqref{eq-M1M2}, we have
\begin{align}\label{eq-f-decrease-1}
&F(\ox^{{k_0} + 1}) \nonumber\\
\leq & F(\ox^{k_0}) - \alpha \left\langle \nabla F(\ox^{k_0}), \od^{k_0} \right\rangle + \frac{L_1 \alpha^2}{2} \|\od^{k_0}\|^2 \nonumber\\
\leq & F(\ox^{k_0}) - \alpha \left\langle \nabla F(\ox^{k_0}), \oB^{k_0} \nabla F(\ox^{k_0}) \right\rangle \nonumber\\
&- \alpha \left\langle \nabla F(\ox^{k_0}), \od^{k_0} - \oB^{k_0} \nabla F(\ox^{k_0}) \right\rangle \\
&+ L_1 \alpha^2 \left(\|\oB^{k_0} \nabla F(\ox^{k_0})\|^2 + \|\od^{k_0} - \oB^{k_0} \nabla F(\ox^{k_0})\|^2 \right)\nonumber\\
\leq & F(\ox^{k_0}) - \frac{\alpha}{M_2} \| \nabla F(\ox^{k_0})\|^2 \nonumber\\
&+ \alpha\left(\frac{1}{4 M_2} \| \nabla F(\ox^{k_0})\|^2
+ M_2 \|\od^{k_0} - \oB^{k_0} \nabla F(\ox^{k_0})\|^2\right) \nonumber\\
&+  L_1 \alpha^2 \left( \frac{1}{M_1^2} \| \nabla F(\ox^{k_0})\|^2 + \|\od^{k_0} - \oB^{k_0} \nabla F(\ox^{k_0})\|^2 \right)\nonumber\\
= & F(\ox^{k_0}) - \left( \frac{3 \alpha}{4 M_2} - \frac{L_1 \alpha^2}{M_1^2} \right) \| \nabla F(\ox^{k_0})\|^2\nonumber\\
& + (M_2 \alpha + L_1 \alpha^2) \|\od^{k_0} - \oB^{k_0} \nabla F(\ox^{k_0})\|^2,\nonumber
\end{align}
where we use  $\frac{1}{M_2} \leq \|\oB^{k_0}\|\leq \frac{1}{M_1}$ in the last inequality.
Next, we bound  $\|\od^{k_0} - \oB^{k_0} \nabla F(\ox^{k_0})\|^2 $. With  \eqref{ineq:inexact_inverse}, we have
\begin{align*}
\begin{split}
\od^k = & \frac{1}{n} \sum_{i = 1}^n B^k_i (g_i^k + r_i^k) \\
= &\frac{1}{n} \sum_{i = 1}^n B^k_i (g_i^k - \og^k) + \oB^k \og^k + \frac{1}{n} \sum_{i = 1}^n B^k_i r^k_i \\
= & \frac{1}{n} \sum_{i = 1}^n (B^k_i - \frac{1}{2 M_1} I_d) (g_i^k - \og^k) + \oB^k\og^k + \frac{1}{n} \sum_{i = 1}^n B^k_i r^k_i
\end{split}
\end{align*}
for all $k$, which implies that
\begin{align}\label{eq-d-Bg}
&\| \od^{k_0} - \oB^{k_0} \nabla F(\ox^{k_0}) \|^2\nonumber\\
= & \Big\|\frac{1}{n} \sum_{i = 1}^n (B^{k_0}_i - \frac{1}{2 M_1} I_d) (g_i^{k_0} - \og^{k_0}) \nonumber\\
&\quad+ \frac{1}{n} \sum_{i = 1}^n B^{k_0}_i r^{k_0}_i+ \oB^{k_0} (\og^{k_0} - \nabla F(\ox^{k_0})) \Big\|^2\nonumber\\
\leq &  4 \Big\| \frac{1}{n} \sum_{i = 1}^n (B^{k_0}_i - \frac{1}{2 M_1} I_d) (g_i^{k_0} - \og^{k_0})\Big\|^2\\
& + 4 \| \oB^{k_0} (\og^{k_0} - \nabla F(\ox^{k_0}))\|^2 +  2 \Big\| \frac{1}{n} \sum_{i = 1}^n B^{k_0}_i r^{k_0}_i\Big\|^2 \nonumber\\
\leq &  \frac{4}{n} \sum_{i = 1}^n \left \| \left(B^{k_0}_i - \frac{1}{2 M_1} I_d\right) (g_i^{k_0} - \og^{k_0})\right\|^2 \nonumber\\
&+ \frac{4}{ M_1^2} \|\og^{k_0} - \nabla F(\ox^{k_0})\|^2 + \frac{2}{n M_1^2} \sum_{i = 1}^n  \| r^{k_0}_i\|^2 \nonumber\\
\leq & \!\frac{1}{n M_1^2}( \|\bg^{k_0}\!\! -\!\! \bW^\infty \bg^{k_0}\|^2 \!+\! 4 L^2_1\|\bx^{k_0} \!-\!\bW^\infty \bx_{k_0}\|^2 \!+\! 2 c_{k_0}^2 \|\bg^{k_0}\|^2 ),\nonumber
\end{align}
where we use  $\|\oB^{k_0}\| \leq \frac{1}{M_1}$, $\|B^{k_0}_i - \frac{1}{2 M_1} I_d\| \leq \frac{1}{2 M_1}$, and \eqref{eq-barg-grad}.
By substituting \eqref{eq-d-Bg} into \eqref{eq-f-decrease-1}, we have
\begin{align}\label{eq-f-f*}
\begin{split}
&F(\ox^{{k_0} + 1}) - F(x^*) \\
\leq & \left(1 - 2\mu\left( \frac{3 \alpha}{4 M_2} - \frac{L_1 \alpha^2}{M_1^2} \right) \right) \left(F(\ox^{{k_0}}) - F(x^*)\right) \\
+ &\frac{M_2 \alpha + L_1 \alpha^2}{n M_1^2}\left(\|\bg^{k_0} - \bW^\infty \bg^{k_0}\|^2 \right.\\
& \left.+ 4 L^2_1\|\bx^{k_0} -\bW^\infty \bx_{k_0}\|^2+2 c_{k_0}^2 \|\bg^{k_0}\|^2\right),
\end{split}
\end{align}
where we use the fact that  $\| \nabla F(\ox^{k})\|^2 \geq 2 \mu (F(\ox^{k}) - F(x^*))$ under Assumption \ref{ass-strongly-cvx}. Further, according to \eqref{ineq:gknorm}, we have
\begin{align}\label{eq-g-2}
&\|\bg^{k}\|^2 \nonumber\\
\leq& ( \|\bg^{k} - \bW^\infty \bg^{k}\| + L_1 \|\bx^{k} - \bW^\infty \bx^{k}\|  + \sqrt{n} \| \nabla F(\ox^{k})\|)^2 \nonumber\\
\leq&  3\|\bg^{k} - \bW^\infty \bg^{k}\|^2 + 3L_1^2 \|\bx^{k} - \bW^\infty \bx^{k}\|  + 3n \| \nabla F(\ox^{k})\|^2 \nonumber\\
\leq&  3\|\bg^{k} - \bW^\infty \bg^{k}\|^2 + 3L_1^2 \|\bx^{k} - \bW^\infty \bx^{k}\| \\
& + 6n L_1(F(\ox^{k}) - F(x^*))\nonumber
\end{align}
for all $k\ge 0$.
Substituting \eqref{eq-g-2} into \eqref{eq-f-f*}, we have
\begin{align}\label{eq-new-02}
&\frac{n}{L_1}\left(F(\ox^{{k_0} + 1}) - F(x^*)\right)\nonumber\\
\leq & \left(1 \!-\! 2\mu\left( \frac{3 \alpha}{4 M_2} \!-\!\frac{L_1 \alpha^2}{M_1^2} \right)\!+\!\frac{12c_{k_0}^2L_1 (M_2 \alpha + L_1 \alpha^2)}{M_1^2} \right)\\
&\!\cdot\!	\frac{n}{L_1}(F(\ox^{{k_0}})\!-\!F(x^*))\nonumber \\
&\quad + \frac{L_1(M_2 \alpha + L_1 \alpha^2)}{M_1^2}  (1 + 6c_{k_0}^2)\cdot\frac{1}{L_1^2}\|\bg^{k_0} - \bW^\infty \bg^{k_0}\|^2 \nonumber\\
&\quad+\frac{(M_2 \alpha + L_1 \alpha^2)L_1}{ M_1^2} (4 + 6c_{k_0}^2)\|\bx^{k_0} -\bW^\infty \bx_{k_0}\|^2.\nonumber
\end{align}
With $\alpha \leq \min\left\{ \frac{M_1^2 (1 - \sigma^2)^3}{100 L_1 M_2{\sigma^{m-1}}},\frac{M_1^2}{200L_1M_2}\right\}$, we have $M_2\alpha+L_1\alpha^2\le 1.01M_2\alpha$  and $c_k \leq \frac{M_1}{4M_2 \sqrt{2\kappa_F}}$. By substituting these inequalities into \eqref{eq-new-02}, we have
$$
\begin{aligned}
&1 - 2\mu\left( \frac{3 \alpha}{4 M_2} - \frac{L_1 \alpha^2}{M_1^2} \right)  + \frac{12c_k^2 L_1 (M_2 \alpha + L_1 \alpha^2)}{M_1^2} \\
\le&1 - \frac{3 \mu\alpha}{2 M_2} + \frac{2\mu L_1 \alpha}{M_1^2}\cdot\frac{M_1}{200L_1}\cdot\frac{M_1}{M_2} + \frac{12c_k^2 L_1 (M_2 \alpha + L_1 \alpha^2)}{M_1^2}\\
\leq& 1 - \frac{2.98\mu \alpha}{2 M_2} + \frac{12L_1 \cdot 1.01 M_2 M_1^2\alpha}{32 M_2^2 M_1^2 L_1/\mu} \\
\leq& 1 - \frac{\mu \alpha}{M_2} = \bJ^{[1]}_{33}
\end{aligned}
$$
and
$$
\begin{aligned}
&\frac{L_1(M_2 \alpha + L_1 \alpha^2)}{M_1^2} (1 + 6c_k^2) \leq \frac{1.01 M_2 \alpha L_1}{ M_1^2} \cdot \left(1+ \frac{6}{32}\right) \\
&\leq \frac{5 M_2 \alpha L_1}{ 4M_1^2} \leq \frac{(1 -\sigma^2)^3}{80\sigma^{m-1} }= \bJ^{[1]}_{32}
\end{aligned}
$$
and
$$
\begin{aligned}
&\frac{(M_2 \alpha + L_1 \alpha^2)L_1}{ M_1^2} (4 + 6c_k^2) \leq \frac{1.01 M_2 \alpha L_1}{M_1^2} \cdot \left(4+ \frac{6}{32}\right) \\
&\leq \frac{5 M_2 \alpha L_1}{ n M_1^2} \leq \frac{(1 -\sigma^2)^3}{20 \sigma^{m-1}}= \bJ^{[1]}_{31}.
\end{aligned}
$$
Thus, we get \eqref{eq-Fx*-J} and complete the proof.
\end{proof}

By combining  Lemmas \ref{lem-3}--\ref{lem-5}, we get \eqref{eq-rate-1}.

Next, we prove \eqref{eq-rate-u} from \eqref{eq-rate-1}. With the parameters satisfying \eqref{eq-a-c}, it is easy to check that
\begin{align*}
\hspace{-1mm}\left(1, \frac{(1 \!-\! \sigma^2)^2}{50}, 2\sigma^{m-1} \right) \bJ^{[1]}
\!\leq\! \left(1 \!-\! \frac{\mu \alpha}{2 M_2} \right) \left(1, \frac{(1 \!-\! \sigma^2)^2}{50}, 2\sigma^{m-1} \right). 
\end{align*}
Thus, by multiplying $\left(1, \frac{(1 - \sigma^2)^2}{50}, 2\sigma^{m-1} \right) $ on both sides of \eqref{eq-rate-1}, we get \eqref{eq-rate-u} and complete the proof of Proposition \ref{prop-1}.
\end{proof}
\subsection{Proof of Proposition \ref{prop-2}}\label{prf-theom-2}

The proof of \eqref{eq-compress-linear} consists of three steps. We are going to  bound the compression error $\|\bE^{k + 1} \|_F$, the difference  $\|\bH^{k + 1} - \tilde{\bH}^{k + 1} \|_F$, and the Hessian tracking error $\|\bH^{k + 1} - \bW^\infty \bH^{k + 1}\|_F$ in Steps I, II, and III, respectively.

\begin{proof}
{\bf Step I:} The following lemma establishes a recursion for the compression error $\|\bE^{k + 1}\|_F$.
\begin{lem}\label{lem:recur-ek}
Under Assumption \ref{asp:compression}, for all $k\ge 0$, we have
\begin{align}
\|\bE^{k + 1}\|_F \leq (1 - \delta)\| \bE^k \|_F + (1 - \delta) \|\bH^k -\tilde{\bH}^k\|_F.
\end{align}
\end{lem}
\begin{proof}
According to Assumption \ref{asp:compression}, we have
\begin{align}\label{eq-expectation-E}
\begin{aligned}
\|\bE^{k + 1}\|_F &\leq (1 - \delta) \| \bE^k + \bH^k -\tilde{\bH}^k\|_F\\
&\le  (1 - \delta) \| \bE^k \|_F + (1 - \delta) \|\bH^k -\tilde{\bH}^k\|_F.
\end{aligned}
\end{align}
This completes the proof.
\end{proof}

{\bf Step II:} The following lemma establish a recursion of the difference $\|\bH^{k + 1} - \tilde{\bH}^{k + 1} \|_F$.
\begin{lem}\label{lem-H-h}
Under Assumptions \ref{asm-W},  \ref{asp:lipschitz-continuous}, and \ref{asp:compression}, for all $k \ge 0$, we have
\begin{align}\label{eq-expect-H-h}
\begin{split}
&\|\bH^{k + 1} - \tilde{\bH}^{k + 1} \|_F\\
\leq&\left(1 -  \delta + 2 \gamma (1 - \delta) \right) \|\bH^k -\tilde{\bH}^k\|_F + 4 \gamma \|\bE^k\|_F\\
& + 2 \gamma \|\bH^k - \bW^\infty \bH^k\|_F + L_2 \|\bx^{k + 1} - \bx^k\|.
\end{split}
\end{align}
\end{lem}

\begin{proof}
According to  Algorithm \ref{alg-1}, we have
\begin{align}\label{eq-30}
\begin{aligned}
&\bH^{k + 1} - \tilde{\bH}^{k + 1} \\
=& \bH^k -\tilde{\bH}^k -  \mathcal{Q} (\bH^k -\tilde{\bH}^k) \\
&- \gamma (I_{nd} - \bW) \hat{\bH}^{k} + \nabla^2 f(\bx^{k + 1}) - \nabla^2 f(\bx^k).
\end{aligned}
\end{align}
Next, we  bound the right-hand side of \eqref{eq-30}. First, according to Assumption \ref{asp:compression}, we have
\begin{align}\label{eq-31}
\begin{split}
\left\|\bH^k -\tilde{\bH}^k -  \mathcal{Q} (\bH^k -\tilde{\bH}^k) \right\|_F \le (1 -  \delta) \|\bH^k -\tilde{\bH}^k\|_F.
\end{split}
\end{align}
Second, according to  Algorithm \ref{alg-1},  we have $$\hat{\bH}^{k} = \bH^k + \bE^k - \bE^{k + 1},$$
which implies that
\begin{align}\label{eq-32}
\begin{split}
&\|(I_{nd} - \bW) \hat{\bH}^{k}\|_F \\
\leq & \|(I_{nd} - \bW) \bH^k\|_F + \|(I_{nd} - \bW) (\bE^k - \bE^{k + 1})\|_F \\		
\leq & 2 \|\bH^k - \bW^\infty \bH^k\|_F + 2 \|\bE^k\|_F + 2\|\bE^{k + 1}\|_F.
\end{split}	
\end{align}
Finally, combining inequalities \eqref{eq-30}--\eqref{eq-32} and using   $\|\nabla^2 f(\bx^{k + 1}) - \nabla^2 f(\bx^k)\| \leq L_2 \|\bx^{k + 1} - \bx^k\|$, we have
\begin{align}\label{eq-H-h}
\begin{split}
& \|\bH^{k + 1} - \tilde{\bH}^{k + 1}\|_F\\
\leq &(1 -  \delta)\|\bH^k -\tilde{\bH}^k\|_F + 2 \gamma \|\bH^k - \bW^\infty \bH^k\|_F \\
&+ 2 \gamma \|\bE^k\|_F+ 2 \gamma \|\bE^{k + 1}\|_F + L_2 \|\bx^{k + 1} - \bx^k\| \\
\leq &\left(1 -  \delta + 2 \gamma (1 - \delta) \right)\|\bH^k -\tilde{\bH}^k\|_F + 4 \gamma \|\bE^k\|_F \\
&+2 \gamma \|\bH^k - \bW^\infty \bH^k\|_F + L_2 \|\bx^{k + 1} - \bx^k\|,
\end{split}
\end{align}
where we use \eqref{eq-expectation-E} in the last inequality. This gives \eqref{eq-expect-H-h} and completes the proof.
\end{proof}

{\bf Step III:} The following lemma bounds the Hessian tracking error  $\|\bH^{k + 1} - \bW^\infty \bH^{k + 1} \|_F$.
\begin{lem}\label{lem-H-cons}
Under Assumptions \ref{asm-W},  \ref{asp:lipschitz-continuous}, and \ref{asp:compression}, if $\gamma<1$, then for all $k$, we have
\begin{align}\label{eq-lem8}
\begin{split}
&\|\bH^{k + 1} - \bW^\infty \bH^{k + 1} \|_F \\
\leq& (1 - \gamma(1 - \sigma))\|\bH^k - \bW^\infty \bH^k\|_F + 4 \gamma \|\bE^k\|_F \\
&+ 2 \gamma (1 - \delta) \|\bH^k -\tilde{\bH}^k\|_F + L_2 \|\bx^{k + 1} - \bx^k\|.
\end{split}
\end{align}
\end{lem}

\begin{proof}
According to $\bH^{k + 1} = \bH^k - \gamma (I_{nd} - \bW) \hat{\bH}^{k} + \nabla^2 f(\bx^{k + 1}) - \nabla^2 f(\bx^k)$, we have
\begin{align}\label{eq-34}
\begin{aligned}
& (I_{nd}- \bW^\infty) \bH^{k + 1} \\
= & (I_{nd}- \bW^\infty)\bH^k - \gamma (I_{nd} - \bW) \hat{\bH}^{k} \\
&+ (I_{nd}- \bW^\infty)( \nabla^2 f(\bx^{k + 1}) - \nabla^2 f(\bx^k)) \\
= & \!\left(I_{nd}\!-\!\bW^\infty\!\!-\!\!\gamma (I_{nd}\!-\!\bW)\right) \bH^k\!-\!\gamma (I_{nd}\!-\!\bW) (\bE^k\!-\!\bE^{k + 1})\\
& + (I_{nd}- \bW^\infty)(\nabla^2 f(\bx^{k + 1}) - \nabla^2 f(\bx^k)),
\end{aligned}		
\end{align}
where the first equality holds because $(I_{nd}- \bW^\infty) (I_{nd} - \bW) = I_{nd} - \bW - \bW^\infty + \bW^\infty = I_{nd} - \bW$ and the second equality holds because $\hat{\bH}^{k} = \bH^k + \bE^k - \bE^{k + 1}$. For the first term on the right-hand side of \eqref{eq-34}, we have
\begin{align}\label{eq-35}
&\|\left(I_{nd}- \bW^\infty - \gamma (I_{nd} - \bW)\right) \bH^k\|\nonumber\\
 = & \|(1 - \gamma) (I_{nd}- \bW^\infty) \bH^k + \gamma (\bW - \bW^\infty) \bH^k\| \\
= & \|(1 - \gamma) (I_{nd}- \bW^\infty) \bH^k + \gamma (\bW - \bW^\infty) (I_{nd}- \bW^\infty) \bH^k\|\nonumber\\
\le& (1 - \gamma + \gamma \sigma) \|(I_{nd}- \bW^\infty) \bH^k\|_F.\nonumber
\end{align}
By taking the Frobenius norm $\|\cdot\|_F$  on both sides of \eqref{eq-34}, we have
\begin{align*}
&\|(I_{nd}-\bW^\infty) \bH^k\|_F \\
\leq &(1 - \gamma(1 - \sigma)) \|\bH^k - \bW^\infty \bH^k\|_F + 2 \gamma \|\bE^k\|_F \\
&+ 2 \gamma \|\bE^{k +1}\|_F + \|\nabla^2 f(\bx^{k + 1}) - \nabla^2 f(\bx^k)\|_F\\
\leq &(1 - \gamma(1 - \sigma)) \|\bH^k - \bW^\infty \bH^k\|_F + 4 \gamma \|\bE^k\|_F \\
&+ 2 \gamma (1 - \delta)\|\bH^k -\tilde{\bH}^k\|_F + L_2 \|\bx^{k + 1} - \bx^k\|,
\end{align*}
where we use \eqref{eq-35} in the first inequality and \eqref{eq-expectation-E} and Assumption 2 in the second inequality. This gives \eqref{eq-lem8} and completes the proof.
\end{proof}

Combining Lemmas \ref{lem:recur-ek}--\ref{lem-H-cons} directly gives \eqref{eq-compress-linear}.

Next, we prove \eqref{eq-linear-v} from \eqref{eq-compress-linear}.	By choosing $\gamma \leq \frac{ \delta^2 (1 - \sigma)}{50}$, it is easy to show that
\begin{align*}
\hspace{-2mm} \left(\frac{ \delta (1 \!-\! \sigma)}{8(1 \!-\! \delta)}, \frac{1 \!-\! \sigma}{4}, 1 \right) \bJ^{[2]} \!
 \leq\! \left(1 \!-\! \frac{\gamma}{2} (1 \!-\! \sigma) \right) \! \left(\frac{ \delta (1 \!-\! \sigma)}{8(1 \!- \!\delta)}, \frac{1 \!-\! \sigma}{4}, 1 \right).
\end{align*}
Thus, by multiplying $\left(\frac{ \delta (1 - \sigma)}{8(1 - \delta)}, \frac{1 - \sigma}{4}, 1 \right)$ on both sides of \eqref{eq-compress-linear}, we get
\begin{align}\label{ineq-2}
u_2^{k + 1} \leq \left(1 - \frac{\gamma}{2} (1 - \sigma) \right) u_2^k + \frac{5L_2}{4}  \|\bx^{k + 1} - \bx^k\|
\end{align}
for all $k\ge 0$.
Further, according to \eqref{eq-23}, we have
\begin{align}\label{ineq-1}
& \|\bx^{{k_0} + 1} - \bx^{k_0} \|^2\nonumber \\ \leq & 2 \left( 2 \!+\! \frac{2\alpha L_1}{M_1} \right)^2 \|\bx^{k_0} \!-\! \bW^\infty \bx^{k_0}\|^2 + \frac{16\alpha^2}{M_1^2} \|\bg^{{k_0}} \!-\! \bW^\infty \bg^{{k_0}}\|^2 \nonumber\\
&+ \frac{32 \alpha^2 L_1}{M_1^2} n \left(F(\ox^{{k_0}}) - F(x^*)\right)\\
\le& 9 \Big(\|\bx^{{k_0}} - \bW^\infty \bx^{{k_0}} \|^2 + \frac{(1 - \sigma^2)^2}{50 L_1^2} \|\bg^{{k_0}} - \bW^\infty \bg^{{k_0}} \|^2\nonumber\\
& + \frac{n}{L_1} \left(F(\ox^{{k_0}}) - F(x^*)\right)\Big)=\frac{9u_1^{k_0}}{\sigma^{m-1}},\nonumber
\end{align}
where we substitute $\alpha$ satisfying \eqref{eq-a-c} in the second inequality. By substituting \eqref{ineq-1} into \eqref{ineq-2}, we get \eqref{eq-linear-v} and complete the proof.
\end{proof}
\subsection{Proof of Proposition \ref{prop-3}}\label{sec-prof-assump-M}
\begin{proof}
We use mathematical induction to prove this proposition. First, it is easy to see that \eqref{eq-M1M2} holds for $k=0$. Second, assume that \eqref{eq-M1M2} holds for all $0, 1, \ldots, k-1$. Then, Proposition \ref{prop-1} implies that
\begin{align}\label{ineq-3}
	u_1^{k} \leq \left(1 - \frac{\mu \alpha}{2 M_2} \right) u_1^{k-1} \leq \cdots \leq \left(1 - \frac{\mu \alpha}{2 M_2} \right)^k u_1^0.
\end{align}
By substituting \eqref{ineq-3} into \eqref{eq-linear-v}, we get
\begin{align}\label{ineq-4}
	\begin{split}
	u_2^{k } \leq& \left(1 - \frac{\gamma}{2} (1 - \sigma) \right) u_2^{k-1} \\
	&+ \frac{15L_2}{4}  \sqrt{\sigma^{-(m-1)}u_1^0} \cdot \left(1 - \frac{\mu \alpha}{2 M_2} \right)^{\frac{k-1}{2}}.
	\end{split}
\end{align}
By unrolling \eqref{ineq-4}, we have
\begin{align}\label{eq-a36}
	u_2^k \leq (u_2^0 - C) \left(1 - \frac{\gamma}{2} (1 - \sigma) \right)^k + C  \left(1 - \frac{\mu \alpha}{4 M_2} \right)^k,
\end{align}
where $C \triangleq \frac{3.75 L_2 \sqrt{\sigma^{-(m-1)}u_1^0}}{\sqrt{1-\frac{\mu\alpha}{2M_2}}-\left(1-\frac{\gamma(1-\sigma)}{2}\right)}.$
Let us define
$$\phi = \max\left\{1 - \frac{\gamma}{2} (1 - \sigma), 1 - \frac{\mu \alpha}{4 M_2}\right\}. $$
Then, \eqref{eq-a36} implies  that
\begin{align}\label{ineq-5}
	u_2^k \leq \phi^k \tilde{u}_2^0
\end{align}		
with
\begin{align}\label{eq-tildeu20}
	\tilde{u}_2^0\triangleq\max\left\{u_2^0-C,C \right\}.
\end{align}	
To complete the proof, it remains to show that  \eqref{eq-M1M2} also holds at time step $k$. To do this, with \eqref{eq-alg-H}, we have
\begin{align}
\oH^k=\frac{1}{n}\sum_{i=1}^{n} \nabla^2 f_i(x_i^k).
\end{align}
A simple computation shows that
\begin{align*}
	\|\oH^k\! -\! \nabla^2 F(\ox^k)\|_F \!\le\! & \frac{1}{n} \sum_{i=1}^{n} \|\nabla^2 f_i(x_i^k) \!-\! \nabla^2 f_i(\ox^k)\|_F \\
	\le & \frac{L_2}{n} \sum_{i=1}^{n} \|x_i^k\! -\! \ox^k\|_F \!\le\! \frac{L_2}{\sqrt{n}} \|\bx^k \!-\! \bW^\infty \bx^k\|_F.
\end{align*}
With the definition of $u_2^k$ and  \eqref{ineq-5}, we have
\begin{align}\label{eq-exp-H-Havr}
&\|H_i^k-\nabla^2 F(\ox^k) \|_F \nonumber \\
\le &\|H_i^k-\oH^k \|_F + \|\oH^k - \nabla^2 F(\ox^k)\|_F \nonumber\\
\le & \|\bH^k - \bW^\infty \bH^k \|_F + \frac{L_2}{\sqrt{n}} \|\bx^k - \bW^\infty \bx^k\|_F.
\end{align}
Based on the definitions of $u_1^k$ and $u_2^k$, \eqref{eq-exp-H-Havr} implies that
\begin{align}
 \|H_i^k-\nabla^2 F(\ox^k)\| \preceq u_2^k  + L_2 \sqrt{\frac{u_1^k}{n}},
\end{align}
where we use the fact that $\|\cdot\|\le \|\cdot\|_F$. Since $ \mu I_d \preceq \nabla^2 F(\ox^k)\preceq L_1 I_d $, we have
\begin{align*}
\left(\mu-L_2 \sqrt{\frac{u_1^k}{n}}- u_2^k\right) I_d \preceq  H_i^k \preceq \left(L_1+L_2 \sqrt{\frac{u_1^k}{n}}+ u_2^k\right) I_d.
\end{align*}
Since $u_1^k \leq u_1^0$, $u_2^k \leq \phi^k \tilde{u}_2^0 \leq \tilde{u}_2^0$, and  $M\geq L_2 \sqrt{\frac{u_1^0}{n}} + \tilde{u}_2^0$, based on the definitions of $M_1$ and $M_2$ given in \eqref{def:M1M2}, we have
\begin{align}
M_1 I_d \preceq  H_i^k+MI_d \preceq  M_2 I_d.
\end{align}
Thus, we prove that \eqref{eq-M1M2} also holds at time step  $k$ and complete the proof.
\end{proof}
\subsection{Proof of Proposition \ref{prop-4}}\label{prf-theom-3}
\begin{proof}
First, we  prove \eqref{eq-a38} in three steps. We are going to  bound the consensus error $\|\bx^{{\tilde{k}_0}} - \bW^\infty \bx^{{\tilde{k}_0}} \|$, the gradient tracking error $
\|\bg^{{\tilde{k}_0}} - \bW^\infty \bg^{{\tilde{k}_0}} \|$, and the network optimality gap $ \|\ox^{{\tilde{k}_0}} - x^* \|$ in Step I, II, and III, respectively. Note that the first two terms have already been bounded in Theorem \ref{theo-1}. {The difference between Theorem \ref{theo-1} and Proposition \ref{prop-4} is that in Proposition \ref{prop-4} we dig deeper into the curvature information contained in the Hessian approximation $\bH^k$ to bound the distance between $\bd^k$ and the true Newton's direction, which gives tighter bounds for $\|\bx^{{\tilde{k}_0}} - \bW^\infty \bx^{{\tilde{k}_0}} \|$ and $
\|\bg^{{\tilde{k}_0}} - \bW^\infty \bg^{{\tilde{k}_0}} \|$ than those given in Lemma \ref{lem-preliminary-x-g}.}

{\bf Step I:} To establish a tighter bound on the consensus error $\|\bx^{{\tilde{k}_0}} - \bW^\infty \bx^{{\tilde{k}_0}} \|$, we need to bound  $\| \bd^{{\tilde{k}_0}} - \bW^\infty \bd^{{\tilde{k}_0}}\|$ on the right-hand side of  \eqref{ineq:xkplus1}.

\begin{lem}\label{lem-d-newton}
Under Assumptions \ref{asm-W} and  \ref{asp:lipschitz-continuous}, if  condition  \eqref{ineq-M1M2} holds for a certain $\tilde{k}_0$ with $\tilde{k}_0\ge K$, then we have
\begin{align}\label{eq-Newtondir}
\begin{split}
& \|\od^{{\tilde{k}_0}} - (\nabla^2 F(\ox^{{\tilde{k}_0}} ))^{-1} \nabla F(\ox^{{\tilde{k}_0}}) \| \\
\leq & \frac{L_1}{\mu \sqrt{n}} \left(1 + \epsilon^{{\tilde{k}_0}} \right) \|\bx^{{\tilde{k}_0}} - \bW^\infty \bx^{{\tilde{k}_0}} \|\\
& + \frac{\epsilon^{{\tilde{k}_0}}}{\mu \sqrt{n}} \left( \|\bg^{{\tilde{k}_0}} - \bW^\infty \bg^{{\tilde{k}_0}}\| + \sqrt{n} L_1 \| \ox^{{\tilde{k}_0}} - x^*\| \right)
\end{split}
\end{align}
and
\begin{align}\label{ineq:dk}
\hspace{-2mm}\|\bd^{{\tilde{k}_0}} \!-\! \bW^\infty \bd^{{\tilde{k}_0}}\|
\!\leq&  \frac{1 \!+\! \epsilon^{{\tilde{k}_0}}}{\mu} \|\bg^{{\tilde{k}_0}} - \bW^\infty \bg^{{\tilde{k}_0}}\|\\
+& \frac{ L_1 \epsilon^{{\tilde{k}_0}}}{\mu} \left(\! \|\bx^{{\tilde{k}_0}} \!-\! \bW^\infty \bx^{{\tilde{k}_0}} \| \!+\! \sqrt{n} \| \ox^{{\tilde{k}_0}} \!-\! x^*\| \!\right).\nonumber
\end{align}
\end{lem}

\begin{proof}	
With $\od^{k}=\frac{1}{n} \sum_{i=1}^{n} B_{i}^{k}\left(g_{i}^{k}+r_{i}^{k}\right),$ we have
\begin{align}\label{eq-45}
\od^{k} = &\frac{1}{n} \sum_{i = 1}^{n} B_i^k  g_i^{k}+ \frac{1}{n} \sum_{i = 1}^{n} B_i^k  r_i^{k}- \frac{1}{n} \sum_{i = 1}^{n} ( \nabla^2 F(\ox^{k}) )^{-1} g_i^{k}\nonumber\\
&+( \nabla^2 F(\ox^{k}))^{-1} \og^{k}
\end{align}
for all $k\ge 0$.
Then, we compute
\begin{align}\label{eq-46}
& \big\| \od^{k} - ( \nabla^2 F(\ox^{k}))^{-1} \nabla F(\ox^{k}) \big\| \nonumber\\
\leq &\big \| ( \nabla^2 F(\ox^{k}) )^{-1} (\og^k \!-\! \nabla F(\ox^{k})) \big\| \!+\!\big \| \od^{k} \!-\! ( \nabla^2 F(\ox^{k}))^{-1} \og^k \big\|\nonumber \\
=&\big \| ( \nabla^2 F(\ox^{k}))^{-1} (\og^k - \nabla F(\ox^{k})) \big\| \nonumber\\
&+\big\|\frac{1}{n} \sum_{i = 1}^{n} B_i^k  g_i^{k}+ \frac{1}{n} \sum_{i = 1}^{n} B_i^k  r_i^{k}- \frac{1}{n} \sum_{i = 1}^{n} ( \nabla^2 F(\ox^{k}))^{-1} g_i^{k}\big\| \nonumber\\
\leq & \frac{L_1}{\mu \sqrt{n}} \|\bx^{k} - \bW^\infty \bx^{k} \| + \frac{1}{n} \sum_{i = 1}^{n} \big\| B_i^k - ( \nabla^2 F(\ox^{k}))^{-1} \big\| \|g_i^{k}\| \nonumber\\
& +{\frac{1}{n} \sum_{i=1}^{n} \|B_i^k \|c_k\|g_i^k\|}
\end{align}
for all $k\ge 0$,
where we use \eqref{eq-45} in the equality and $\og^k=\frac{1}{n}\sum_{i=1}^n f_i(x_i)$ in the last inequality. To bound the second term on the right-hand side of \eqref{eq-46}, with the fact that $\|A^{-1} - B^{-1}\| \leq \|A^{-1}\| \|B^{-1}\| \|A - B\|$, we have
\begin{align}\label{eq-44}
\begin{split}
&\big\|( H_i^{\tilde{k}_0} )^{-1} - ( \nabla^2 F(\ox^{{\tilde{k}_0}}))^{-1} \big\|\\
 \leq& \frac{1}{\mu M_1} \left\| H^{{\tilde{k}_0}}_i -  \nabla^2 F(\ox^{{\tilde{k}_0}}) \right\|\\
\leq& \frac{1}{\mu M_1} \Big( \big\| H^{{\tilde{k}_0}}_i  - \oH^{\tilde{k}_0}\big \|+\big\| \oH^{\tilde{k}_0} - \nabla^2 F(\ox^{{\tilde{k}_0}}) \big\|\Big),\\
\end{split}
\end{align}	
where we use $B_i^{\tilde{k}_0}=( H_i^{\tilde{k}_0})^{-1}$ and $M_1 I_d \preceq H_i^{\tilde{k}_0}\preceq M_2 I_d$.

Next, we bound $\left\| \oH^k- \nabla^2 F(\ox^{k}) \right\|$ on the right-hand side of \eqref{eq-44}. According to  $\bH^{k + 1} = \bH^k - \gamma (I_{nd} - \bW) \hat{\bH}^{k} + \nabla^2 f(\bx^{k + 1}) - \nabla^2 f(\bx^k)$ and the initialization $H_i^0=\nabla^2 f_i(\bx^0)$, we know that
\begin{align}\label{eq-H-DAC}
\oH^{k} = \overline{ \nabla^2 f} (\bx^{k})
\end{align}
for all $k\ge 0$.
Then, according to Assumption \ref{asp:lipschitz-continuous}, \eqref{eq-H-DAC} implies that
\begin{align}\label{eq-H-DAC-1}
\|\oH^{k}- \nabla^2 F(\ox^{k})\|_F \leq \frac{L_2}{\sqrt{n}} \|\bx^{k} - \bW^\infty \bx^{k} \|
\end{align}
for all $k\ge 0$.
Substituting \eqref{eq-H-DAC-1} into \eqref{eq-44}, we have
\begin{align}\label{eq-H-nabla2}
\begin{split}
&\big\|( H_i^{\tilde{k}_0} )^{-1} - ( \nabla^2 F(\ox^{{\tilde{k}_0}}))^{-1} \big\|\\
\leq&\frac{1}{\mu M_1}\Big( \big\| H^{{\tilde{k}_0}}_i  -\oH^{\tilde{k}_0}\big \|+  \frac{L_2}{\sqrt{n}} \|\bx^{{\tilde{k}_0}} - \bW^\infty \bx^{{\tilde{k}_0}} \|\Big).
\end{split}
\end{align}	
Substituting \eqref{eq-H-nabla2} into \eqref{eq-46}, we have
\begin{align}\label{eq-H-nabla2-1}
& \big\| \od^{{\tilde{k}_0}} - ( \nabla^2 F(\ox^{{\tilde{k}_0}}))^{-1} \nabla F(\ox^{{\tilde{k}_0}}) \big\| \nonumber\\
\leq & \frac{L_1}{\mu \sqrt{n}} \|\bx^{{\tilde{k}_0}} - \bW^\infty \bx^{{\tilde{k}_0}} \| +\frac{1}{n} \sum_{i=1}^{n} \|B_i^{\tilde{k}_0} \|c_{\tilde{k}_0}\|g_i^{\tilde{k}_0}\|\nonumber\\
&+ \frac{1}{\mu M_1 n} \sum_{i = 1}^{n} \| H_i^{\tilde{k}_0} -  \nabla^2 F(\ox^{{\tilde{k}_0}}) \| \|g_i^{{\tilde{k}_0}}\|\\
\leq & \frac{L_1}{\mu \sqrt{n}} \|\bx^{{\tilde{k}_0}} - \bW^\infty \bx^{{\tilde{k}_0}} \| + \frac{1}{\mu M_1 n} \Big( L_2 \|\bx^{{\tilde{k}_0}} - \bW^\infty \bx^{{\tilde{k}_0}} \| \nonumber\\
&+ \|\bH^{{\tilde{k}_0}} - \bW^\infty \bH^{{\tilde{k}_0}}\|_F+{\mu c_{\tilde{k}_0}}\sqrt{n} \Big) \|\bg^{{\tilde{k}_0}}\|.\nonumber
\end{align}
In addition, following similar steps as in the derivation of \eqref{eq-H-nabla2-1}, we have
\begin{align}\label{eq-47}
&\| \bd^{{\tilde{k}_0}} - \bW^\infty \bd^{{\tilde{k}_0}}\|\nonumber\\
\leq & \left\| (I_{nd} - \bW^\infty) \left(\bd^{{\tilde{k}_0}} - (\diag\{\nabla^2 F(\ox^{{\tilde{k}_0}})\} )^{-1} \bg^{{\tilde{k}_0}} \right) \right\| \nonumber\\
&+ \left\| (I_{nd} - \bW^\infty) ( \diag\{\nabla^2 F(\ox^{{\tilde{k}_0}})\})^{-1} \bg^{{\tilde{k}_0}}\right\| \\
\leq & \frac{\|\bg^{{\tilde{k}_0}}\|}{\mu M_1 \sqrt{n}} \Big( L_2 \|\bx^{{\tilde{k}_0}} - \bW^\infty \bx^{{\tilde{k}_0}} \| + \|\bH^{{\tilde{k}_0}} - \bW^\infty \bH^{{\tilde{k}_0}}\|_F \nonumber\\
& +{\mu c_{\tilde{k}_0}}  \sqrt{n} \Big)
+ \frac{1}{\mu} \|\bg^{{\tilde{k}_0}} - \bW^\infty \bg^{{\tilde{k}_0}}\|\nonumber \\
= & \frac{\epsilon^{\tilde{k}_0}}{\mu} \|\bg^{{\tilde{k}_0}}\| + \frac{1}{\mu} \|\bg^{{\tilde{k}_0}} - \bW^\infty \bg^{{\tilde{k}_0}}\|.\nonumber
\end{align}
By substituting \eqref{ineq:gknorm} into \eqref{eq-H-nabla2-1} and \eqref{eq-47}, we get \eqref{eq-Newtondir} and \eqref{ineq:dk}. This completes the proof.
\end{proof}

With Lemma \ref{lem-d-newton}, the following corollary gives a tighter bound on $ \|\bx^{k + 1} - \bW^\infty \bx^{k + 1} \|$.

\begin{coro}\label{coro-2}
Under the setting of Lemma \ref{lem-d-newton},  we have
\begin{equation*}
\begin{split}
& \|\bx^{{\tilde{k}_0} + 1} \!-\! \bW^\infty \bx^{{\tilde{k}_0} \!+\! 1} \| \\
\!\leq\! & \bJ^{[3]}_{11}\! \|\bx^{\tilde{k}_0} \!-\! \bW^\infty \bx^{\tilde{k}_0}\| \!+\! \frac{\bJ^{[3]}_{12}}{L_1}\|\bg^{\tilde{k}_0} \!-\! \bW^\infty \bg^{\tilde{k}_0}\|
 \!+\! \bJ^{[3]}_{13} \sqrt{n} \|\ox^{\tilde{k}_0} \!-\! x^*\|.
\end{split}
\end{equation*}
\end{coro}
\begin{proof}
We substitute the tighter bound on $\| \bd^{{\tilde{k}_0}} - \bW^\infty \bd^{{\tilde{k}_0}}\|$ given in \eqref{ineq:dk} into \eqref{ineq:xkplus1} and complete the proof.
\end{proof}

{\bf Step II:} To get a tighter bound on the gradient tracking error $\|\bg^{k} - \bW^\infty \bg^{k} \|$, we need to bound  $\| \bx^{k+1} - \bx^k\|$ on the right hand of \eqref{ineq:gkplus1} by taking advantage of the curvature information.

\begin{lem}\label{lem:xkplus1-xk}
Under Assumptions \ref{asm-W}--\ref{ass-strongly-cvx}, if condition \eqref{eq-M1M2} holds for a certain ${\tilde{k}_0}$, then we have
\begin{align}\label{eq-a49}
&\|\bx^{{\tilde{k}_0} + 1} - \bx^{\tilde{k}_0} \|\nonumber\\
 \leq & \left( 2 + \alpha \kappa_F + 2\alpha \kappa_F \epsilon^{\tilde{k}_0} \right) \|\bx^{\tilde{k}_0} - \bW^\infty \bx^{\tilde{k}_0}\| \nonumber\\
 &+ \alpha \kappa_F(\sigma^m + 2\epsilon^{\tilde{k}_0}) \cdot \frac{1}{L_1} \|\bg^{{\tilde{k}_0}} - \bW^\infty \bg^{{\tilde{k}_0}}\| \\
& + \alpha \left(1 + 2\kappa_F \epsilon^{\tilde{k}_0} + \frac{L_2}{2\mu} \|\ox^{\tilde{k}_0} - x^*\| \right) \sqrt{n} \|\ox^{\tilde{k}_0} - x^*\|.\nonumber
\end{align}
\end{lem}

\begin{proof}
According to $\bx^{k + 1} = \bW^m  (\bx^k - \alpha \bd^k)$, we have
\begin{align}\label{eq-50}
\begin{aligned}
&\|\bx^{k + 1} - \bx^k \|\\
\leq & \|(\bW^m - I_{nd}) (\bx^k - \bW^\infty \bx^k)\| + \alpha \|\bW^m \bd^k\|\\
\leq &2 \|\bx^k - \bW^\infty \bx^k\| + \alpha \|\bW^m \bd^k\|
\end{aligned}
\end{align}
for all $k\ge 0$.
Next, we bound the second term $	\|\bW^m \bd^k\|$ on the right-hand side of \eqref{eq-50}. According to the triangle inequality, we have
\begin{align}\label{eq-51}
&\|\bW^m \bd^k\| \nonumber\\
\leq& \|\bW^m \bd^k \!-\! \bW^\infty \bd^k\| + \sqrt{n} \|\od^k \!-\! (\nabla^2 F(\ox^{k} ))^{-1} \nabla F(\ox^{k}) \| \nonumber\\
&+ \sqrt{n} \|(\nabla^2 F(\ox^{k} ))^{-1} \nabla F(\ox^{k}) \|
\end{align}
for all $k\ge 0$.
For the third term on the right-hand side of \eqref{eq-51}, according to the fact  that
\begin{align}\label{centra-newton}
	\|\ox^k - x^* - (\nabla^2 F(\ox^{k}))^{-1} \nabla F(\ox^{k})\| \leq \frac{L_2}{2\mu} \|\ox^k - x^*\|^2
\end{align}
holds for all $k\ge 0$, we have
\begin{align}
\hspace{-4mm}	\|(\nabla^2 F(\ox^{k} ))^{-1} \nabla F(\ox^{k}) \| \leq \|\ox^k - x^*\| +  \frac{L_2}{2\mu} \|\ox^k - x^*\|^2
\end{align}
for all $k\ge 0$.
For the first term on the right-hand side of \eqref{eq-51}, we have $\|\bW^m \bd^k - \bW^\infty \bd^k\| \leq \sigma^m \|\bd^k - \bW^\infty \bd^k\|$ and then use \eqref{ineq:dk} to bound it. For the second term on the right-hand side of \eqref{eq-51}, we  use \eqref{eq-Newtondir} to bound it.
Thus, \eqref{eq-51} implies that
\begin{align}\label{ineq:dknorm}
\|\bW^m \bd^{{\tilde{k}_0}}\|
\leq & \kappa_F (1 + 2\epsilon^{\tilde{k}_0}) \|\bx^{{\tilde{k}_0}} - \bW^\infty \bx^{{\tilde{k}_0}}\| \nonumber\\
&+  \kappa_F(\sigma^m + 2\epsilon^{\tilde{k}_0}) \cdot \frac{1}{L_1} \|\bg^{{\tilde{k}_0}} - \bW^\infty \bg^{{\tilde{k}_0}}\| \\
+&  \left(1 + 2\kappa_F \epsilon^{\tilde{k}_0} + \frac{L_2}{2\mu} \|\ox^{\tilde{k}_0} - x^*\| \right) \sqrt{n} \|\ox^{\tilde{k}_0} - x^*\|.\nonumber
\end{align}
Finally, substituting \eqref{ineq:dknorm} into \eqref{eq-50}, we get \eqref{eq-a49} and
complete the proof.
\end{proof}

With the tighter bound on $	\|\bx^{{\tilde{k}_0} + 1} - \bx^{\tilde{k}_0} \|$ given in Lemma \ref{lem:xkplus1-xk}, we have the following corollary, which gives a tighter bound on the gradient tracking error $\|\bg^{{\tilde{k}_0} + 1} - \bW^\infty \bg^{{\tilde{k}_0} + 1} \|$.

\begin{coro}\label{coro-3}
Under the setting of Lemma \ref{lem:xkplus1-xk}, we have
\begin{equation*}
\begin{split}
&\frac{1}{L_1} \|\bg^{{\tilde{k}_0} + 1} - \bW^\infty \bg^{{\tilde{k}_0} + 1} \| \\
\leq & \bJ^{[3]}_{21} \|\bx^{\tilde{k}_0} - \bW^\infty \bx^{\tilde{k}_0}\| + \bJ^{[3]}_{22}\frac{1}{L_1} \|\bg^{\tilde{k}_0} - \bW^\infty \bg^{\tilde{k}_0}\|
\\
&+ \bJ^{[3]}_{23} \sqrt{n} \|\ox^{\tilde{k}_0} \!-\! x^*\|.
\end{split}
\end{equation*}
\end{coro}

\begin{proof}
We substitute \eqref{eq-a49}  into \eqref{ineq:gkplus1}   and complete the proof.
\end{proof}

{\bf Step III:} The following corollary bounds $ \|\ox^{k + 1} - x^* \|$ based on the locally quadratic convergence of the centralized Newton's method.

\begin{coro}\label{lem-xopt}
Under the setting of Lemma \ref{lem:xkplus1-xk}, we have
\begin{equation*}
\begin{split}
& \|\ox^{{\tilde{k}_0} + 1} - x^* \| \\
\leq & \bJ^{[3]}_{33}\|\ox^{\tilde{k}_0} - x^*\| + \bJ^{[3]}_{31} \|\bx^{{\tilde{k}_0}} - \bW^\infty \bx^{{\tilde{k}_0}} \|\\
& + \bJ^{[3]}_{32} \frac{1}{L_1}  \|\bg^{{\tilde{k}_0}} - \bW^\infty \bg^{{\tilde{k}_0}}\|.
\end{split}
\end{equation*}
\end{coro}

\begin{proof}
With $\ox^{k + 1} = \ox^k - \alpha \od^k$,  we have
\begin{align}\label{eq-56}
\begin{aligned}
\|\ox^{k + 1} - x^* \|\leq&\|\ox^k - x^* - \alpha (\nabla^2 F(\ox^{k}))^{-1} \nabla F(\ox^{k})\|\\
& + \alpha \|\od^k - (\nabla^2 F(\ox^{k}))^{-1} \nabla F(\ox^{k})\|
\end{aligned}
\end{align}
for all $k\ge 0$.
With \eqref{centra-newton}, which is given by the centralized Newton's method \cite{nocedal2006numerical}, we get
\begin{align}\label{eq-57}
\begin{aligned}
&\|\ox^k - x^* - \alpha (\nabla^2 F(\ox^{k}))^{-1} \nabla F(\ox^{k})\| \\
\leq& (1 - \alpha) \|\ox^k - x^*\| + \frac{\alpha L_2}{2\mu} \|\ox^k - x^*\|^2 \\
=& (1 - \alpha + \alpha \delta^k) \|\ox^k - x^*\|
\end{aligned}
\end{align}
for all $k\ge 0$.
By substituting \eqref{eq-57} and \eqref{eq-Newtondir} into \eqref{eq-56}, we  complete the proof.	
\end{proof}

Combining Corollaries \ref{coro-2}, \ref{coro-3}, and  \ref{lem-xopt}, we get \eqref{eq-a38}.
Next, we prove \eqref{eq-a39} from \eqref{eq-a38}. With $\alpha = 1$, the coefficient matrix $\bJ^{[3]}$ in \eqref{eq-a38} can be simplified as
\begin{equation*}
\bJ^{[3]}=\mathcal{E}_1+ \mathcal{E}_2
\end{equation*}
with
\begin{equation*}
\mathcal{E}_1\triangleq \begin{bmatrix}
\sigma^m  & \sigma^m \kappa_F & 0 \\
\sigma^m \left( 2 + \kappa_F \right) & \sigma^m \left(1 + \kappa_F \sigma^m \right) & \sigma^m \\
\kappa_F & 0 & 0
\end{bmatrix}
\end{equation*}
and
\begin{equation*}
\begin{aligned}
\mathcal{E}_2 &\triangleq \begin{bmatrix}
\sigma^m \kappa_F \epsilon^{\tilde{k}_0} & \sigma^m \epsilon^{\tilde{k}_0} \kappa_F & 2 \sigma^m \epsilon^{\tilde{k}_0} \kappa_F \\
2 \sigma^m \kappa_F \epsilon^{\tilde{k}_0} & 2 \sigma^m \kappa_F \epsilon^{\tilde{k}_0}  & \sigma^m \left( 2\kappa_F \epsilon^{\tilde{k}_0} + \delta^{\tilde{k}_0} \right) \\
\kappa_F\epsilon^{{\tilde{k}_0}} & \kappa_F \epsilon^{{\tilde{k}_0}}& \delta^{\tilde{k}_0} +  \kappa_F\epsilon^{{\tilde{k}_0}}
\end{bmatrix} \\
&\leq (\kappa_F \epsilon^{\tilde{k}_0} + \delta^{\tilde{k}_0}) \begin{bmatrix}
\sigma^m  & \sigma^m  & 2 \sigma^m \\
2\sigma^m & 2\sigma^m  & 2\sigma^m \\
1 & 1 & 1
\end{bmatrix}.
\end{aligned}
\end{equation*}
By defining $z\triangleq [1;\sigma^{-\frac{m}{4}};0.5 \sigma^{-\frac{3m}{4}}]$, we claim that
\begin{align}\label{eq-local-rate-1}
\mathcal{E}_1 z \le \frac{3\sigma^{\frac{m}{2}}}{4}z
\end{align}
and
\begin{align}\label{eq-local-rate-2}
\mathcal{E}_2 z \le \frac{\sigma^{\frac{m}{2}}}{4}z.
\end{align}

Indeed, to prove \eqref{eq-local-rate-1},
with  $m > \frac{ 4 \log (4\kappa_F)}{ -\log \sigma}$, or equivalently, $\kappa_{F}\le \frac{\sigma^{-\frac{m}{4}}}{4}$, it is easy to show that
\begin{equation*}
\mathcal{E}_1\left(1; \sigma^{-\frac{m}{4}} ; 0.5\sigma^{-\frac{3m}{4}}\right)\le \frac{3\sigma^{\frac{m}{2}}}{4}\left(1; \sigma^{-\frac{m}{4}}; 0.5\sigma^{-\frac{3m}{4}}\right).
\end{equation*}
To prove \eqref{eq-local-rate-2},
it is easy to show that
\begin{equation*}
\begin{bmatrix}
\sigma^m  & \sigma^m  & 2 \sigma^m \\
2\sigma^m & 2\sigma^m  & 2\sigma^m \\
1 & 1 & 1
\end{bmatrix}
\begin{bmatrix}
1 \\ \sigma^{-\frac{m}{4}} \\ 0.5\sigma^{-\frac{3m}{4}}
\end{bmatrix} \leq {5} \begin{bmatrix}
1 \\ \sigma^{-\frac{m}{4}} \\ 0.5\sigma^{-\frac{3m}{4}}
\end{bmatrix}.
\end{equation*}
Under the condition
\begin{align}\label{ineq:changecondition}
	\kappa_F \epsilon^{\tilde{k}_0} + \delta^{\tilde{k}_0} \leq \frac{1}{{20}}\sigma^{\frac{m}{2}},
\end{align}
we have
$$\mathcal{E}_2 z \leq \frac{1}{4}\sigma^{\frac{m}{2}} z.$$
Therefore, by summing up \eqref{eq-local-rate-1} and \eqref{eq-local-rate-2}, we get \eqref{eq-a39} and complete the proof.
\end{proof}

\subsection{Proof of Proposition \ref{prop-5}}\label{prf-prop-5}
\begin{proof}
According to the definition of $\epsilon^k$ and $\delta^k$ given in \eqref{eq-J3}, we have
\begin{align}\label{eq-82}
&\kappa_F \epsilon^k + \delta^k \nonumber\\
 =& \frac{\kappa_F}{M_1} \left( \frac{1}{\sqrt{n}}L_2 \|\bx^{k} \!-\! \bW^\infty \bx^{k} \|\! +\! \frac{1}{\sqrt{n}} \|\bH^{k} \!-\! \bW^\infty \bH^{k}\|_F+{\mu c_k}\right)\nonumber \\
 &+ \frac{ L_2}{2\mu} \|\ox^k - x^*\| \nonumber\\
 \leq& \frac{\kappa_F}{M_1\sqrt{n}} u_2^k + \frac{\kappa_F L_2}{M_1\sqrt{n}} u_3^k + \frac{1}{{40}}\sigma^{\frac{m}{2}}
\end{align}
for all $k\ge 0$, where the inequality holds because $c_k\le \frac{M_1\sigma^{m/2}}{40\mu \kappa_F}$, $\|\bH^{k} - \bW^\infty \bH^{k}\|_F \leq u_2^k$, and
\begin{align*}
	\frac{\kappa_F}{M_1} \frac{1}{\sqrt{n}}L_2 \|\bx^{k} - \bW^\infty \bx^{k} \| + \frac{ L_2}{2\mu} \|\ox^k - x^*\| \leq \frac{\kappa_FL_2}{M_1\sqrt{n}} u_3^k.
\end{align*}
On the other hand, it is worth noting that  \eqref{ineq-2} holds for any $k\ge 0$, i.e.,
\begin{align}
	u_2^{k + 1} \leq \left(1 - \frac{\gamma}{2} (1 - \sigma) \right) u_2^k + \frac{5L_2}{4}  \|\bx^{k + 1} - \bx^k\|.
\end{align}
With Theorem \eqref{theo-1}, we know that \eqref{eq-11} holds for any $k\ge 0$. Thus, by substituting $\alpha=1$ into \eqref{eq-11}, we have
\begin{align}
		\|\bx^{k + 1} - \bx^k \|
		&\leq \left( 2 + \frac{2 L_1}{M_1} \right) \|\bx^k - \bW^\infty \bx^k\|\\
		&\hspace{-6mm} + \frac{2}{M_1} \|\bg^{k} - \bW^\infty \bg^{k}\| + \frac{2 \sqrt{n}}{M_1} \| \nabla F(\ox^{k})\| \leq \frac{4L_1}{M_1} u_3^k,\nonumber
\end{align}
where the last inequality holds because $\| \nabla F(\ox^{k})\| \leq L_1\|\ox^{k} - x^*\|$. Define
$$
A_k \triangleq u_2^k + \frac{5L_1 L_2}{M_1(1-\sigma^{m/2})} u_3^k.
$$
Then, \eqref{eq-82} implies that
\begin{align}\label{ineq-6}
	\kappa_F \epsilon^k + \delta^k \leq  \frac{\kappa_F}{M_1\sqrt{n}} A_k + \frac{1}{{40}}\sigma^{\frac{m}{2}}
\end{align}
for all $k\ge 0$.
The motivation behind the definition of the sequence $\{ A_k\}_{k\ge 0}$ is given as follows. If Proposition \ref{prop-4} holds at time step $k$, i.e., $u_3^{k+1} \leq \sigma^{m/2}u_3^k$, then we have
\begin{align}\label{ineq-Ak}
\begin{aligned}
	A_{k+1} = &u_2^{k+1} + \frac{5L_1 L_2}{M_1(1-\sigma^{m/2})} u_3^{k+1}\\
	\leq & u_2^k + \frac{5L_2}{4} \cdot  \frac{4L_1}{M_1} u_3^k + \frac{5L_1 L_2}{M_1(1-\sigma^{m/2})} \sigma^{m/2}  u_3^k\\
	=& u_2^k + \frac{5L_1 L_2}{M_1(1-\sigma^{m/2})}  u_3^k = A_k.
\end{aligned}
\end{align}

Now, we are ready to prove that \eqref{ineq-M1M2} and \eqref{ineq-eps-delta} hold for all $k\ge 0$ by induction. First, we show that \eqref{ineq-M1M2} and \eqref{ineq-eps-delta} hold at time step $K$.
Since $F(\ox^{k}) - F(x^*) \geq
\frac{\mu}{2}\|\ox^{k}-\bx^*\|^2$, we know that $\bq_1^K \geq \frac{1}{2\kappa_F} (\bq_3^K)^2$. Thus, we have
\begin{align*}
	(u_3^K)^2 
	\leq & \left(1, \sigma^{-\frac{m}{2}}, 0.25 \sigma^{-\frac{3m}{2}}\right)(\bq_3^K)^2\\
	\leq & 2\kappa_F \Big(1, \sigma^{-\frac{m}{2}}, 0.25 \sigma^{-\frac{3m}{2}}\Big) \bq_1^K \\
	\leq & \frac{100\kappa_F\sigma^{-\frac{5m}{2}}}{(1 - \sigma^2)^2} \left(1, \frac{(1 - \sigma^2)^2}{50}, 2\sigma^{m-1} \right) \bq_1^K\\
	 = & \frac{100\kappa_F\sigma^{-\frac{5m}{2}}}{(1 - \sigma^2)^2} u_1^K,
\end{align*}
which implies that
\begin{align*}
	A_K = & u_2^K + \frac{5L_1 L_2}{M_1(1-\sigma^{m/2})} u_3^K\\
	\leq & u_2^K + \frac{5L_1 L_2}{M_1(1-\sigma^{m/2})} \cdot \frac{10\sqrt{\kappa_F}\sigma^{-\frac{5m}{4}}}{1 - \sigma^2} \sqrt{u_1^K}\\
	\leq & \tilde{u}_2^0 \phi^K + \frac{50L_1 L_2 \sqrt{\kappa_F}\sigma^{-\frac{5m}{4}}}{M_1(1-\sigma^{m/2})(1 - \sigma^2)}  \left(1 - \frac{\mu \alpha}{2 M_2} \right)^{\frac{K}{2}} \sqrt{u_1^0}\\
	\leq & \left(\tilde{u}_2^0  + \frac{50L_1 L_2 \sqrt{\kappa_F}\sigma^{-\frac{5m}{4}}}{M_1(1-\sigma^{m/2})(1 - \sigma^2)} \sqrt{u_1^0}\right) \phi^K. 
\end{align*}
Here, the inequality holds because $\left(1 - \frac{\mu \alpha}{2 M_2} \right)^{\frac{1}{2}} \leq 1 - \frac{\mu \alpha}{4 M_2} \leq \phi$.

Further, with \eqref{ineq-6}, we have
$$\kappa_F \epsilon^K + \delta^K \leq \frac{\kappa_F}{M_1\sqrt{n}} A_K + \frac{1}{{40}}\sigma^{\frac{m}{2}}.
$$
Thus,  condition \eqref{ineq-eps-delta} holds at time step $K$ if
\begin{align}\label{eq-AK}
\frac{\kappa_F}{M_1} A_K &\leq \frac{\kappa_F}{M_1}  \left(\tilde{u}_2^0  + \frac{50L_1 L_2 \sqrt{\kappa_F}\sigma^{-\frac{5m}{4}}}{M_1(1-\sigma^{m/2})(1 - \sigma^2)} \sqrt{u_1^0}\right) \phi^K \nonumber\\
&\leq \frac{1}{{40}}\sigma^{\frac{m}{2}},
\end{align}
which is equivalent to
\begin{align}\label{ineq:conditon_on_K}
K \geq& \frac{\log \frac{\sigma^{m/2}}{\frac{40\kappa_F}{M_1}  \left(\tilde{u}_2^0  + \frac{50L_1 L_2 \sqrt{\kappa_F}\sigma^{-\frac{5m}{4}}}{M_1(1-\sigma^{m/2})(1 - \sigma^2)} \sqrt{u_1^0}\right)}}{\log \phi} \\
=& \frac{\frac{m}{2}\log \sigma - \log \frac{40\kappa_F}{M_1\sqrt{n}}  \left(\tilde{u}_2^0  + \frac{50L_1 L_2 \sqrt{\kappa_F}\sigma^{-\frac{5m}{4}}}{M_1(1-\sigma^{m/2})(1 - \sigma^2)} \sqrt{u_1^0}\right)}{\log \phi}.\nonumber
\end{align}
Besides, based on \eqref{eq-exp-H-Havr} and the definition of $A_k$, we have
\begin{align}\label{eq-HAK}
		&\|H_i^K-\nabla^2 F(\ox^K) \|_F \nonumber \\
		\leq & \|\bH^K - \bW^\infty \bH^K \|_F + \frac{L_2}{\sqrt{n}} \|\bx^K - \bW^\infty \bx^K\|_F \nonumber\\
\leq & u_2^K \leq A_K \leq \frac{M_1}{40} = \frac{\mu}{41},
\end{align}
where we use \eqref{eq-AK} in the last inequality. Then, \eqref{eq-HAK} implies that
$$
M_1I_d =(u - \frac{\mu}{41})I_d \leq H_i^K \leq (L_1 + \frac{\mu}{41})I_d = M_2I_d.
$$
Thus, both \eqref{ineq-M1M2} and \eqref{ineq-eps-delta} hold at time step $K$.

Second, assume that \eqref{ineq-M1M2} and \eqref{ineq-eps-delta} hold for $K, \ldots, k-1$. Then, according to Proposition \ref{prop-4}, we have the inequality $u_3^{k+1} \leq \sigma^{m/2}u_3^k$ for all $K, \ldots, k-1$. Thus, \eqref{ineq-Ak} shows that
\begin{align}\label{eq-A-descent}
	A_k \leq A_{k-1} \leq \ldots \leq A_K,
\end{align}
which,  together with \eqref{ineq-6}, implies that
$$\kappa_F \epsilon^k + \delta^k \!\leq\! \frac{\kappa_F}{M_1\sqrt{n}} A_k + \frac{1}{{40}}\sigma^{\frac{m}{2}} \!\leq \!\frac{\kappa_F}{M_1\sqrt{n}} A_K + \frac{1}{{40}}\sigma^{\frac{m}{2}}\! \leq\! \frac{1}{{20}}\sigma^{\frac{m}{2}}.
$$
Therefore, we know that $\eqref{ineq-eps-delta}$ also holds at time step $k$, Besides, by the definition of $A_k$, we have
\begin{align}
	\|H_i^k-\nabla^2 F(\ox^k) \|_F \leq u_2^k \leq A_k \leq A_K \leq \frac{M_1}{40} = \frac{\mu}{41},
\end{align}
which implies that
$$
 M_1I_d =(u - \frac{\mu}{41})I_d \leq H_i^k \leq (L_1 + \frac{\mu}{41})I_d = M_2I_d.
$$
Thus, both \eqref{ineq-M1M2} and \eqref{ineq-eps-delta} hold at time step $k$. This completes the mathematical induction.

Finally,
by substituting $M_1= \frac{40\mu}{41}$ into \eqref{ineq:conditon_on_K}, we get
$$
	K \geq \frac{\frac{m}{2}\log \sigma - \log \frac{41\kappa_F}{\mu\sqrt{n}}  \left(\tilde{u}_2^0  + \frac{52L_2\kappa_F  \sqrt{\kappa_F}\sigma^{-\frac{5m}{4}}}{(1-\sigma^{m/2})(1 - \sigma^2)} \sqrt{u_1^0}\right)}{\log \phi}
$$
and complete the proof.
\end{proof}

\bibliographystyle{IEEEtran}
\bibliography{IEEEabrvPnADMM}
\end{document}